\documentclass{amsart}
\usepackage{amsmath,amsfonts,amssymb,amscd,verbatim,multicol}
\usepackage[arrow,matrix,cmtip]{xy}

\usepackage{fullpage}

\newcommand{\trdeg}{\operatorname{tr.deg}}

\newcommand{\bpr}{\begin{proof}}
\newcommand{\epr}{\end{proof}}

\newcommand{\spec}{\operatorname{Spec}}
\newcommand{\hspec}{\operatorname{hSpec}}

\newcommand{\mc}{\mathcal}
\newcommand{\mf}{\mathfrak}
\newcommand{\cha}{\operatorname{char}}
\newcommand{\mb}{\mathbb}

\newcommand{\dra}{\dashrightarrow}

\newcommand{\GK}{\operatorname{GK}}

\newcommand{\wt}{\widetilde}

\newcommand{\aut}{\operatorname{Aut}}

\DeclareMathOperator{\HB}{H}

\newcommand{\ord}{countable-avoiding}
\newcommand{\hord}{countable-avoiding}

\newcommand{\cl}{\operatorname{Cl}}

\newcommand{\naive}{na{\"\i}ve}

\newcommand{\beq}{\begin{equation}}
\newcommand{\eeq}{\end{equation}}

\newcommand{\Aut}{{\rm Aut}}

\newcommand{\bbar}[1]{\overline{#1}}

\renewcommand{\Lsh}{\mathcal{L}}

\newcommand{\ZZ}{{\mathbb Z}}

\newcommand{\PP}{{\mathbb P}}

 \DeclareMathOperator{\Pic}{Pic}
\DeclareMathOperator{\pic}{Pic}

 \DeclareMathOperator{\Ker}{Ker}

\numberwithin{equation}{section}
 \theoremstyle{plain}
\newtheorem{theorem}[equation]{Theorem}

\newtheorem{lemma}[equation]{Lemma}

\newtheorem{corollary}[equation]{Corollary}
\newtheorem{proposition}[equation]{Proposition}

\theoremstyle{definition}
\newtheorem{question}[equation]{Question}
\newtheorem{definition}[equation]{Definition}
\newtheorem{notation}[equation]{Notation}

\newtheorem{remark}[equation]{Remark}

\newtheorem{standing-hypothesis}[equation]{Standing Hypothesis}
\newtheorem{example}[equation]{Example}

\newtheorem{conjecture}[equation]{Conjecture}

\begin{document}
\title{The Dixmier-Moeglin equivalence for twisted homogeneous coordinate rings}
\author{J. Bell, D. Rogalski, and S. J.  Sierra}

\address{(Bell)
Department of Mathematics, Simon Fraser University, 8888 University
Drive, Burnaby, BC,  V5A 1S6, Canada.} \email{jpb@math.sfu.ca}

\address{(Rogalski)
Department of Mathematics, University of California San Diego, La
Jolla, CA 92093-0112, USA.} \email{drogalsk@math.ucsd.edu}

\address{(Sierra)
Department of Mathematics, University of Washington, Box 354350,
Seattle, WA 98195-4350, USA.} \email{sjsierra@math.washington.edu}

\thanks{The first author was partially supported by NSERC through grant
611456, the second author was partially supported by the NSF through
grant DMS-0600834, and the third author was partially supported by
the NSF through grant DMS-0802935.}



\begin{abstract}
Given a projective scheme $X$ over a field $k$, an automorphism
$\sigma: X \to X$, and a $\sigma$-ample invertible sheaf $\mc{L}$,
one may form the twisted homogeneous coordinate ring $B = B(X,
\mc{L}, \sigma)$, one of the most fundamental constructions in
noncommutative projective algebraic geometry.  We study the
primitive spectrum of $B$, as well as that of other closely related
algebras such as skew and skew-Laurent extensions of commutative
algebras. Over an algebraically closed, uncountable field $k$ of
characteristic zero, we prove that that the primitive ideals of $B$
are characterized by the usual Dixmier-Moeglin conditions whenever
$\dim X \leq 2$.
\end{abstract}

\date{\today}

\maketitle

\section{Introduction}
Let $A$ be an algebra over a base field $k$. To better understand
$A$, one of the basic properties in which one is interested is the
structure of the simple right $A$-modules. Often an explicit
understanding of the simple modules is difficult, and one settles
for less exact information: one aims instead to understand the
primitive ideals of $A$, that is, the possible annihilators of the
simple modules.  For this reason, questions about primitive ideals
are clearly fundamental, and there is a long history in
noncommutative ring theory of their study.

One of the seminal achievements in the subject is the work of
Dixmier and Moeglin on the characterization of the primitive
spectrum of the universal enveloping algebra $\mf{U}(L)$ of a
finite-dimensional Lie algebra $L$ over $\mb{C}$. This work showed
the that the primitive ideals can be recognized among all prime
ideals in the prime spectrum $\spec \mf{U}(L)$ in two different
ways.  A prime ideal of an algebra $A$ is called \emph{locally
closed} if $\{ P \}$ is an open subset of its closure in the Zariski
topology on $\spec A$. Assuming $A$ is also Goldie, the prime P is
called \emph{rational} if the center of the Goldie quotient ring
$Q(A/P)$ is an algebraic extension of $k$.  If $L$ is as above and
 $A = \mf{U}(L)$, then a prime $P$ of $A$ is primitive if and only if
it is locally closed, if and only if it is rational. Thus any Goldie
algebra $A$ for which both of these characterizations hold is said
to satisfy the \emph{Dixmier-Moeglin (DM)-equivalence.}

More recent research has extended the work on enveloping algebras in
two directions.  First, there are results which show that some of the
implications among the properties of being primitive, rational, and
locally closed (as well as some other related properties) hold
 under quite general hypotheses on an algebra $A$.
The paper \cite{RoSm} has a good review of what general implications
are currently known; we also review in Section~\ref{prim-primes-sec}
below some of the ones we need. Second, many other interesting
special classes of rings have been shown to satisfy the
DM-equivalence.  For example, the DM-equivalence holds for the most
familiar quantized coordinate rings in the theory of quantum groups
\cite[Corollary II.8.5]{BG}. Typically, in such classes of examples
where the DM-equivalence has been proved, there is some connection
of the algebras in question to geometry (in the quantized coordinate
ring case, there is a torus action.)

In this paper, we consider an important class of graded rings with a
strong connection to projective algebraic geometry, and study the
DM-equivalence for them; in particular we prove that the
DM-equivalence always holds in small-dimensional cases. The main
examples of interest are twisted homogeneous coordinate rings, which
are fundamental in the theory of noncommutative projective geometry.
Given a projective $k$-variety $X$, an invertible sheaf $\mc{L}$ on
$X$ and an automorphism $\sigma: X \to X$, one defines from this
data the \emph{twisted homogeneous coordinate ring} $B(X, \mc{L},
\sigma)$, as follows. Set $\mc{L}_0 = \mc{O}_X$ and for each $n \geq
1$, put $\mc{L}_n = \mc{L} \otimes \sigma^* \mc{L} \otimes \dots
\otimes (\sigma^{n-1})^* \mc{L}$. Then $B(X, \mc{L}, \sigma) =
\bigoplus_{n \geq 0} \HB^0(X, \mc{L}_n)$ is a naturally an
$\mb{N}$-graded $k$-algebra under the multiplication $f \star g = f
\cdot (\sigma^m)^*g$, where $f \in B_m$, $g \in B_n$. We only make
this construction when $\mc{L}$ is \emph{$\sigma$-ample}, an
additional technical condition which we will define in the body of
the paper. This condition ensures, among other things, that $B$ is
noetherian. See \cite{AV} and \cite{Ke1} for more background on
these rings.

In fact, our methods for studying primitivity will apply not just to
the rings $B = B(X, \mc{L}, \sigma)$ but also to several other
graded algebras.  One of the fundamental properties of the ring $B$,
in case $X$ is integral, is that its graded quotient ring has the
form $Q_{\rm gr}(B) = k(X)[t, t^{-1}; \sigma]$, where $k(X)$ is the
field of rational functions on $X$, and $\sigma: k(X) \to k(X)$ is
the induced automorphism.  We call any $\mb{Z}$-graded Ore domain
$A$ where $(Q_{\rm gr}(A))_0$ is a field \emph{birationally
commutative}, and in general one expects the properties of such
rings to be closely connected to commutative geometry. In
particular, this class also includes the skew and skew-Laurent
extensions $R[t; \sigma]$ and $R[t, t^{-1}; \sigma]$, where $R$ is a
commutative domain, finitely generated over $k$, and $\sigma: R \to
R$ is an automorphism.  The primitivity of skew and skew-Laurent
extensions has been studied by several authors, but our methods put
these examples in an interesting new context.

We are now ready to state our main theorem.
\begin{theorem}
\label{main-thm} {\rm [Theorem~\ref{main-DM-thm}]} Let $k$ be an
uncountable algebraically closed field of characteristic $0$.
\begin{enumerate}
\item Let $X$ be a projective variety over $k$ with automorphism $\sigma: X \to
X$ and $\sigma$-ample invertible sheaf $\mc{L}$.  If $\dim X \leq
2$, then $B(X, \mc{L}, \sigma)$ satisfies the Dixmier-Moeglin
equivalence.
\item Let $S$ be a commutative finitely generated $k$-algebra,
with automorphism $\sigma: S \to S$.  Let $A$ be either of the
algebras $S[t; \sigma]$ or $S[t, t^{-1}; \sigma]$.  If $\dim S \leq
2$ and $\GK A < \infty$, then $A$ satisfies the Dixmier-Moeglin
equivalence.
\end{enumerate}
\end{theorem}
We remark that part (2) of the theorem is false without the
assumption that $\GK A < \infty$, as several well-known examples
which fail the DM-equivalence are skew-Laurent rings of this type.
Also, the assumption that $\mc{L}$ is $\sigma$-ample in part (1)
forces $\GK B < \infty$, whereas twisted homogeneous coordinate
rings over non-$\sigma$-ample sheaves typically have exponential
growth. Thus the two parts of the theorem are more analogous than
they at first appear.  We also see no obstruction to this theorem
being true in all dimensions, and we conjecture that this is so. Our
current methods are highly dependent, however, on results specific
to the surface case.

Now we begin to describe the structure of the paper, and the methods
we employ, several of which are interesting in themselves. In
Section~\ref{prim-primes-sec}, we review some well-known results and
show that in the special case of a birationally commutative
$\mb{Z}$-graded algebra $A$, to prove the Dixmier-Moeglin
equivalence it often suffices to examine the \emph{homogeneous}
prime spectrum. Specifically, under mild extra assumptions
(including the uncountability of $k$), the DM-equivalence will hold
for $A$ unless $A$ has a prime graded factor algebra whose number of
homogeneous height one primes is countably infinite.  We then
introduce in Section~\ref{examples-sec} the main examples $A$ we
will consider, and show that in each case there is a
quasi-projective scheme $X$, with automorphism $\sigma: X \to X$,
such that $(X, \sigma)$ encodes all of the information about the
homogeneous primes of the algebra.   For $A = B(X, \mc{L}, \sigma)$,
for example, the homogeneous primes of $A$ are in one-to-one
inclusion-reversing correspondence with \emph{$\sigma$-irreducible}
closed subsets $Z$ of $X$---the closed $\sigma$-invariant subsets
$Z$ such that $\sigma$ acts as a cycle on the irreducible components
of $Z$. Then if $X$ is integral, the homogeneous height one primes
of $A$ correspond to the \emph{maximal $\sigma$-irreducible} closed
subsets, in other words, the maximal elements among the set of all
proper $\sigma$-irreducible subsets of $X$.   A similar
correspondence holds for the skew and skew-Laurent extensions as in
Theorem~\ref{main-thm}(2), taking $X = \spec S$.

In Sections~\ref{orbital-char-sec} and \ref{base-ext-sec}, we study
the possible structure of the maximal $\sigma$-irreducible subsets
of a variety with automorphism $\sigma$, and prove the following
very general geometric result, which is of independent interest.
\begin{theorem}
\label{main-geom-thm} {\rm [Theorem~\ref{unc-irr-thm},
Theorem~\ref{inv-divisor-thm}]} Let $k$ be uncountable and
algebraically closed, and let $X$ be a quasi-projective integral
$k$-scheme with automorphism $\sigma: X \to X$.  Then the following
are equivalent:
\begin{enumerate}
\item $X$ has uncountably many maximal $\sigma$-irreducible closed subsets.

\item $X$ has infinitely many maximal $\sigma$-irreducible closed subsets of
codimension $1$ in $X$.

\item $X$ has a nonconstant  rational function $f \in k(X)$ such that $\sigma(f) = f$.
\end{enumerate}
\end{theorem}
\noindent We note that given the existence of a rational function
$f$ as in part (3) of the theorem, the fibers of the induced
rational map $f:  X \dra \PP^1$ give a cover of $X$ by uncountably
many $\sigma$-invariant codimension-1 closed subsets; it is
surprising to find that this is forced by each of the seemingly much
weaker conditions in parts (1) and (2).

The previous theorem allows us, for one, to reformulate the study of
the cardinality of the set of maximal $\sigma$-invariant subsets in
a more elegant way.  Given $x \in X$, we write $\mc{O}_x = \{
\sigma^n(x) |n \in \mb{Z} \}$.  We say that $(X, \sigma)$ is
\emph{ordinary} if for all $\sigma$-irreducible closed subsets $Z
\subseteq X$, the set $\{z \in Z | \mc{O}_z\ \text{is Zariski dense
in Z} \}$ is open in $Z$. In the other words, ordinary automorphisms
are those for which the property of lying on a dense orbit is an
open condition, inside  any closed subset $Z$ to which the
automorphism restricts.  We show in Section~\ref{orbital-char-sec},
using the equivalence between (1) and (3) in
Theorem~\ref{main-geom-thm}, that any of our examples will satisfy
the DM-equivalence as long as the corresponding geometric data $(X,
\sigma)$ is ordinary. In addition, the equivalence between (2) and
(3) in Theorem~\ref{main-geom-thm} is very useful in the study of
which pairs $(X, \sigma)$ are ordinary, as it shows that any failure
of ordinariness must be happening in codimension at least $2$.  The
proof of this latter equivalence in Theorem~\ref{main-geom-thm}
involves a reduction to the case of a countable base field, and so
we also study more generally in Section~\ref{base-ext-sec} how the
structure of the $\sigma$-invariant closed subsets is affected by
changing the base field.

In Section~\ref{growth-type-sec}, we restrict to the case that $X$
is integral with $\dim X \leq 2$.  We show how the growth of our
birationally commutative algebra $A$ is related to the geometric
properties of its associated data $(X, \sigma)$. In case $\GK A <
\infty$, we say that $(X, \sigma)$ has \emph{finite growth type},
and we show in particular that this is equivalent to the condition
$\GK k(X)[t, t^{-1}; \sigma] < \infty$.  (If $X$ is projective,
finite growth type is also equivalent to $\sigma$ being
\emph{quasi-unipotent} as in \cite{Ke1}.)  These results extend some
material from \cite{Ro}, which also relies on a dynamical
classification of birational maps of surfaces by Diller and Favre
\cite{DF}.

In Section~\ref{ord-auto-sec}, we prove that various special kinds
of automorphisms of varieties are ordinary, and in particular show
the following result:
\begin{theorem}
\label{main-ord-thm} {\rm [Theorem~\ref{dim2-ord-thm}]} Let $k$ be
uncountable, algebraically closed and of characteristic $0$.  Let
$X$ be a quasi-projective variety over $k$ with $\dim X = 2$, and
let $\sigma: X \to X$ be an automorphism of finite growth type. Then
$(X, \sigma)$ is ordinary.
\end{theorem}
\noindent The proof of this theorem depends on many non-trivial
results from the theory of projective surfaces.  The main result
given above, Theorem~\ref{main-thm}, essentially immediately
follows; the proof is given in Section~\ref{summary-sec}, where we
also give a number of illustrative examples.  Among other things, we
review several examples  of skew and skew-Laurent extensions that do
not satisfy the DM-equivalence.

\section*{Acknowledgments}
This paper benefitted greatly from conversations with Lance Small
and Ken Goodearl, and we give them many thanks.

\section{Primitivity and the prime spectrum}

\label{prim-primes-sec}

Let $A$ be a $k$-algebra.  The following hypothesis is crucial for
the results we are about to review.
\begin{standing-hypothesis}
\label{k-hyp} Throughout this section, $k$ is an uncountable base
field.
\end{standing-hypothesis}
\noindent Later sections of the paper will require varying
assumptions on $k$, and so the appropriate standing hypothesis on
$k$ will be stated at the beginning of each section, as well as in
the statements of major results.

We first recall some standard results relating the primitivity of a
prime ideal $P$ in $A$ to the cardinality of the set of height one
primes in $A/P$.  We then specialize to the case of $\mb{Z}$-graded
birationally commutative algebras $A$.  Such algebras may have many
non-homogeneous primes, but we show the somewhat surprising result
that to detect the Dixmier-Moeglin equivalence, it typically
suffices to examine only the homogeneous prime spectrum of $A$.

Recall from the introduction that an algebra $A$ satisfies the
Dixmier-Moeglin (DM)-equivalence if the sets of locally closed,
rational, and primitive ideals of $A$, as defined there, all
coincide.   We say that a prime ideal $P$ of a prime algebra $A$ has
\emph{height one} if there does not exist a prime ideal $Q$ with
$(0) \subsetneq Q \subsetneq P$.   Let $\spec A$ be the set of all
prime ideals of $A$, considered as a topological space using the
Zariski topology, whose closed sets are the sets $V(I) = \{ P \in
\spec A | P \supseteq I \}$ as $I$ ranges over the ideals of $A$.
For the definitions of other standard notions we use here, such as
Jacobson ring, we refer the reader to \cite{MR}.  It is not obvious,
in general, that a nonzero prime ideal of a noetherian algebra $A$
must contain a height one prime.  To ensure this we will usually
assume that $A$ has DCC on prime ideals, which will be easy to show
in all of the examples in which we are interested.  It is a
well-known open question whether a noetherian ring must have DCC on
prime ideals.

The next two results, Lemmas~\ref{height-one-lem} and \ref{PI-lem},
are well-known, but for lack of a single reference we sketch the
proofs for the convenience of the reader.  The first result shows
how the DM-equivalence for the prime $(0)$ comes for free, given a
cardinality assumption on the set of height one primes.
\begin{lemma}
\label{height-one-lem}  Let $A$ be a prime noetherian, countably
generated $k$-algebra with DCC on prime ideals.
\begin{enumerate}
\item Suppose that $A$ has finitely many height one primes.  Then $(0)$
is locally closed, primitive, and rational.
\item Suppose that $A$ has uncountably many height one primes.  Then
$(0)$ is not locally closed, not primitive, and not rational.
\end{enumerate}
\end{lemma}
\begin{proof}
Since $A$ is noetherian and $\dim_k A$ has smaller cardinality than
$|k|$, it is a standard result that $A$ is a Jacobson ring
satisfying the Nullstellensatz over $k$ \cite[Proposition II.7.12,
Proposition II.7.16]{BG}.  In this case, for any prime $P$ we have
$P$ locally closed $\implies$ $P$ primitive $\implies$ $P$ rational
\cite[Lemma II.7.15]{BG}.

(1) Suppose that $A$ has finitely many height one primes, say
$\{P_1, P_2, \dots, P_n \}$.  By the previous paragraph, we just
need to prove that $(0)$ is locally closed.  Since every prime ideal
contains a height one prime, we have $\{(0) \} = \spec A \setminus
V(I)$, where $I = P_1 \cap P_2 \cap \dots \cap P_n \neq 0$.  Thus
$\{(0) \}$ is open in its closure, $\spec A$.

(2) Suppose that $A$ has uncountably many height one primes.  By the
first paragraph of the proof, we need show only that $(0)$ is not
rational.  Suppose that $A$ has a countable separating set of
ideals, in other words there exist nonzero ideals $\{I_1, I_2, I_3,
\dots \}$ of $A$ with the following property: for any nonzero ideal
$J$ of $A$, there is some $n \geq 1$ so that $J \supseteq I_n$.  Each height one prime
$P$ contains $I_n$ for some $n$, and since $P/I_n$ is a minimal
prime of the noetherian ring $A/I_n$, we see that at most finitely
many height one primes contain a given $I_n$.  This contradicts the
assumption that $A$ has uncountably many height 1 primes, so $A$
does not have a countable separating set of ideals.  Since $A$ is
noetherian and countably generated, Irving's theorem \cite[Theorem
2.2]{Ir} now applies and shows that $(0)$ is not rational.
\end{proof}

The preceding lemma shows that the existence of algebras with a
countably infinite number of height one primes is the main
obstruction to the DM-equivalence.  This suggests the following
definition.
\begin{definition} \label{height-one-def}
Given a $k$-algebra $A$, we say that $\spec A$ is
\emph{countable-avoiding} if every prime factor algebra of $A$ has
either finitely many or uncountably many height one primes.
\end{definition}

The following is an immediate consequence of the preceding
definition and Lemma~\ref{height-one-lem}.
\begin{corollary}
\label{ord-DM-cor}   Let $A$ be a noetherian, countably generated
$k$-algebra such that $A$ has DCC on prime ideals and $\spec A$ is
\ord. Then $A$ satisfies the DM-equivalence.
\end{corollary}

The preceding corollary applies to many standard kinds of algebras,
for example algebras satisfying a polynomial identity (PI algebras).
\begin{lemma}
\label{PI-lem} Let $R$ be a noetherian PI algebra which is countably
generated over $k$.  Then $R$ has DCC on prime ideals, $\spec R$ is
\ord, and $R$ satisfies the DM-equivalence.
\end{lemma}
\begin{proof} The fact that $R$ has DCC on prime ideals is
a theorem of Small \cite{Sm}.  By Corollary~\ref{ord-DM-cor}, we
need only show that every prime factor ring $R/P$ has a finite or
uncountable number of height one primes.  Since the hypotheses pass
to factor rings, we may assume that $R$ is prime and prove that $R$
has either finitely many or uncountably many height one primes.

Suppose first that $R$ is primitive.  By Kaplansky's Theorem, a
primitive PI ring is a central simple algebra, so $(0)$ is the only
prime of $R$ in this case.

If $R$ is not primitive, let $Z$ be the center of $R$.  By Posner's
Theorem, $RS^{-1}$ is a finite-dimensional central simple $L =
ZS^{-1}$-algebra, where $S$ is the multiplicative system $Z
\setminus \{ 0 \}$.  If $L/k$ is algebraic, then $RS^{-1}$ is an
algebraic $k$-algebra and hence $R$ itself is as well.  But then
every regular element of $R$ is already a unit, and so $RS^{-1} =
R$. In this case, $R$ is simple and so certainly primitive, a
contradiction. Thus $L/k$ is transcendental.  Now suppose that $R$
has at most countably many height one primes $\{ P_1, P_2, \dots
\}$.  Each such prime $P_i$ contains a regular central element $z_i
\in P_i$.  If $T$ is the multiplicative system generated by the
elements $\{ z_1, z_2, \dots \}$, then $RT^{-1}$ is already simple
(since every prime contains a height one prime) and so $RT^{-1} =
RS^{-1}$. But then $RS^{-1}$ is countably generated over $k$, in
particular countable dimensional over $k$.  This is a contradiction
because $k$ is uncountable and $L/k$ is transcendental, which forces
$L$ and hence $RS^{-1}$ to be uncountable dimensional over $k$. Thus
$R$ has uncountably many height one primes.
\end{proof}

In this paper, we are primarily interested in studying primitivity
for certain classes of $\mb{Z}$-graded algebras.  Let $A =
\bigoplus_{n \in \mb{Z}} A_n$ be a prime $\mb{Z}$-graded noetherian
$k$-algebra. Then the graded quotient ring $Q = Q_{\rm gr}(A)$ of
$A$ exists by \cite[Theorem 1]{GS}, and by \cite[Theorem I.5.8]{NV},
$Q$ must have the form $M_n(S)$  (with some choice of grading) for
some graded division ring $S$.  Moreover, either $A = A_0$ is
trivially graded and $S = S_0 = D$ is a division ring in degree $0$,
or else $S \cong D[t, t^{-1}; \sigma]$ for some division ring $D$ in
degree $0$ and element $t$ of positive degree \cite[Corollary
I.4.3]{NV}. We call $A$ \emph{birationally PI} if $D$ is PI. In
practice, all of the specific examples we study later will actually
satisfy the stronger condition that either $Q_{\rm gr}(A) = K$ or
$Q_{\rm gr}(A) \cong K[t, t^{-1}; \sigma]$ for some commutative ring
$K$ and element $t$ of positive degree, in which case we say that
$A$ is \emph{birationally commutative}.

Birationally PI algebras are special in the following way.
\begin{lemma}
\label{bc-dichotomy-lem} Let $A$ be a prime $\mb{Z}$-graded
noetherian $k$-algebra which is birationally PI, and let $Q = Q_{\rm
gr}(A)$.
\begin{enumerate}
\item $Q$ is either PI or a simple ring.
\item If $Q$ is simple,  then every nonzero prime ideal of $A$
contains a nonzero homogeneous prime ideal.
\end{enumerate}
\end{lemma}
\begin{proof}
(1) If $A = A_0$ is trivially graded, then by assumption $Q = D$ is
a PI division ring, so that case is certainly fine.  Thus assume
that $Q \cong M_n(D[t, t^{-1}; \sigma])$ with some choice of
grading, where $D$ is PI. Suppose first that some positive power of
$\sigma$ is inner, say $\sigma^n(x) = a^{-1} x a$ for some $0 \neq a
\in D$ and $n > 0$.  Then $D[t, t^{-1}; \sigma]$ is finitely
generated as a module over the subring $D[t^n, t^{-n}; \sigma^n]
\cong D[z, z^{-1}]$, where $z = a t^n$. So $D[z, z^{-1}]$ is PI
\cite[Corollary 13.1.11]{MR}, and thus $D[t, t^{-1}; \sigma]$ is PI
\cite[Corollary 13.4.9]{MR}.  Then $Q = M_n(D[t, t^{-1}; \sigma])$
is also PI \cite[Theorem 13.4.8]{MR}. Otherwise, no power of
$\sigma$ is inner. In this case $D[t, t^{-1}; \sigma]$ is simple by
\cite[Theorem 1.8.5]{MR}, so $Q$ is also simple.

(2) Since the graded quotient ring $Q$ of $A$ is simple, every
nonzero prime ideal $P$ of $A$ contains a homogeneous regular
element $0 \neq a \in P$.  Let $\wt{P}$ be the homogeneous ideal
generated by all homogeneous elements in $P$. If $I$ and $J$ are
homogeneous ideals such that $IJ \subseteq \wt{P}$, then $IJ
\subseteq P$ and so $I \subseteq P$ or $J \subseteq P$. But since
$I$ and $J$ are homogeneous, this means that $I \subseteq \wt{P}$ or
$J \subseteq \wt{P}$ by the definition of $\wt{P}$; we conclude that
$\wt{P}$ is a nonzero homogeneous prime ideal contained in $P$.
\end{proof}

In the next proposition, we justify the claim made at the beginning
of this section, namely that to prove the Dixmier-Moeglin
equivalence for birationally PI algebras, it often is enough to
study only the homogeneous prime spectrum.  For this purpose, we
need the following homogeneous analog of
Definition~\ref{height-one-def}.
\begin{definition}
Let $A$ be a $\mb{Z}$-graded algebra, and let $\hspec A$ be the
homogeneous prime spectrum with the Zariski topology.  If $A$ is
prime, a homogeneous prime $P$ of $A$ has \emph{h-height one} if
there does not exist a homogeneous prime $Q$ with $(0) \subsetneq Q
\subsetneq P$.  We say that $\hspec A$ is \emph{\hord} if every
factor ring $B = A/P$ by a homogeneous prime $P$ has either finitely
many or uncountably many h-height one primes.
\end{definition}

\begin{proposition}
\label{hord-DM-prop}  Let $A$ be a noetherian, countably generated
$\mb{Z}$-graded algebra such that  for every homogeneous prime $P$
of $A$, $A/P$ is birationally PI;  $A$ has DCC on homogeneous
primes; and $\hspec A$ is \hord. Then $A$ has DCC on primes, $\spec
A$ is \ord, and $A$ satisfies the Dixmier-Moeglin equivalence.
\end{proposition}
\begin{proof}
Since the minimal primes of $A$ are homogeneous \cite[Corollary
C.I.1.9]{NV}, we may assume without loss of generality that $A$
itself is prime.

We first prove that $A$ has DCC on prime ideals.  By noetherian
induction, we can assume that all proper homogeneous prime factor
rings of $A$ have DCC on prime ideals.  Applying
Lemma~\ref{bc-dichotomy-lem} to the ring $A$, we see that either $A$
is PI or else every nonzero prime $P$ of $A$ contains a nonzero
homogeneous prime $\wt{P}$.  In the former case, $A$ has DCC on
prime ideals by Lemma~\ref{PI-lem}, so assume the latter case.
Recall from the proof of Lemma~\ref{bc-dichotomy-lem} that $\wt{P}$
is the ideal generated by the homogeneous elements in $P$.  Thus a
descending chain of nonzero prime ideals of $A$, $P_1 \supseteq P_2
\supseteq P_3 \supseteq \dots$, leads to a descending chain of
nonzero homogeneous primes $\wt{P}_1 \supseteq \wt{P}_2 \supseteq
\dots$ which stabilizes by assumption, with $\wt{P}_n = \wt{P}_{n+1}
= \dots$, say.  Now $A/\wt{P}_n$ is a proper homogeneous prime
factor ring of $A$ and so has DCC on primes by the induction
hypothesis.  Thus the original chain $P_1 \supseteq P_2 \supseteq
\dots$ of primes, all of which contain $\wt{P}_n$, must also
stabilize.

Now we prove that $\spec A$ is \ord.  Let $M$ be any prime ideal of
$A$, and let $P \subseteq M$ be a homogeneous prime ideal maximal
among the homogeneous primes contained in $M$.  Applying
Lemma~\ref{bc-dichotomy-lem} to the birationally PI ring $A/P$, we
have two cases to consider. If $Q = Q_{\rm gr}(A/P)$ is PI, then
$A/P$ is PI, and hence $A/M$ is PI.  Then $A/M$ has finitely many or
uncountably many height one primes by Lemma~\ref{PI-lem}. Otherwise,
$Q$ is simple and every nonzero prime of $A/P$ contains a nonzero
homogeneous prime. In particular, in this case $M = P$ by choice of
$P$, so $M$ is itself homogeneous. Moreover, since every nonzero
prime of $A/M$ contains a nonzero homogeneous prime, it follows that
the h-height one primes of $A/M$ are the same as the height one
primes of $A/M$.  Thus the assumption that $\hspec A$ is \hord\
implies that $A/M$ has finitely many or uncountably many height one
primes.  Since $M$ was an arbitrary prime, we conclude that $\spec
A$ is \ord.

Finally, $A$ satisfies the Dixmier-Moeglin equivalence by
Corollary~\ref{ord-DM-cor}.
\end{proof}

\section{Some examples of birationally commutative algebras}
\label{examples-sec}

Beginning with this section, we will work with graded algebras whose
homogeneous ideals can be described using algebraic geometry. For
simplicity, we adopt the following blanket convention on the base
field.
\begin{standing-hypothesis} \label{k-hyp2} In this section,
$k$ stands for an uncountable, algebraically closed field.
\end{standing-hypothesis}

We begin now to introduce some important classes of examples to
which Proposition~\ref{hord-DM-prop} applies. The examples in which
we are especially interested are twisted homogeneous coordinate
rings, because these rings are so important in the theory of
noncommutative projective geometry. We will see that for such rings
$B$, the condition that $\hspec B$ is \hord, as defined in the
previous section, corresponds to a subtle algebro-geometric
condition. We also show that this same geometric condition can be
used to study the more familiar examples of skew and skew-Laurent
extensions of commutative algebras, which will shed new light on the
primitive spectra of these examples.

\begin{example}
\label{thcr-ex} Given a projective $k$-scheme $X$, a
$k$-automorphism $\sigma$ of $X$, and an invertible sheaf $\mc{L}$
on $X$, we defined in the introduction the \emph{twisted homogeneous
coordinate ring} $B(X, \mc{L}, \sigma) = \bigoplus_{n = 0}^{\infty}
\HB^0(X, \mc{L}_n)$, where here $\mc{L}_n = \mc{L} \otimes
\sigma^*\mc{L} \otimes (\sigma^2)^* \mc{L} \otimes \dots \otimes
(\sigma^{n-1})^*\mc{L}$ for each $n$. We always assume when making
this construction that $\mc{L}$ is \emph{$\sigma$-ample}: this means
that for every coherent sheaf $\mc{F}$ on $X$, $\HB^i(X, \mc{F}
\otimes \mc{L}_n) = 0$ for all $n \gg 0$ and $i > 0$.
\end{example}

The $\sigma$-ample hypothesis is necessary in order for the algebra
$B(X, \mc{L}, \sigma)$ to have well-behaved properties.  In
particular, in this case $B$ is noetherian, and its homogeneous
prime spectrum is completely determined by the following lemma.
\begin{lemma}
\label{thcr-primes-lem} Let $B = B(X, \mc{L}, \sigma)$ be a twisted
homogeneous coordinate ring as in Example~\ref{thcr-ex}.  Suppose
that $Y \subseteq X$ is a reduced closed subscheme such that
$\sigma(Y) = Y$ and $\sigma$ acts on the irreducible components of
$Y$ in a single cycle.  If $\mc{J}$ is the ideal sheaf on $X$
corresponding to $Y$, then $\bigoplus_{n = 0}^{\infty} \HB^0(X,
\mc{J} \otimes \mc{L}_n)$ is a homogeneous prime ideal of $B$.
Conversely, every homogeneous prime ideal of $B$ is of this form.
\end{lemma}
\begin{proof}
The proof of the first assertion is straightforward; we prove the
second. Let $J$ be a homogeneous prime ideal of $B$.  By
\cite[Theorem~3.12]{AV} and \cite[Lemma~4.4]{AS1}, there is an ideal
sheaf $\mc{J}$ on $X$ such that $\sigma^* \mc{J} = \mc{J}$ and, for
some $N$, we have $J_n = H^0(X, \mc{J} \otimes \Lsh_n)$ for $n \geq
N$. Clearly,
\[ J \subseteq \bigoplus_{n \geq 0} H^0(X, \mc{J} \otimes \Lsh_n)= J'.\]
 If $Y$ is the subscheme of $X$ defined by $\mc{J}$,
it is easy to see that $Y$ is reduced and that $\sigma$ acts on the components of $Y$ in a single
cycle.

Let $b \in J'_m$ for some $m \geq 1$.  As $\mc{J}$ is
$\sigma$-invariant, $(bB)^N \subseteq (J')_{\geq N} = J_{\geq N}.$
Primeness of $J$ implies that $b \in J$; thus $J = J'$.
\end{proof}

Given the previous result, the structure of the homogeneous prime
spectrum of a twisted homogeneous coordinate ring (with a
$\sigma$-ample sheaf) is a purely geometric concern, and the study
of the geometry involved will occupy us for much of the later
sections of the paper.  Since we will use it several times later, we  note here that it
follows easily from the definition that $\sigma$-ampleness restricts
to any $\sigma$-invariant subscheme.  More specifically, if $Z$ is a
subscheme of $X$ with $\sigma(Z) = Z$, then given a $\sigma$-ample
sheaf $\mc{L}$ on $X$, $\mc{L} \vert_Z$ is $\sigma \vert_Z$-ample.

The next examples are likely to be quite familiar to the reader.
\begin{example}
\label{skew-ex} Let $S$ be a commutative finitely generated
$k$-algebra, with $k$-algebra automorphism $\sigma: S \to S$. Let $U
= S[t; \sigma]$ be the skew polynomial ring and let  $T = S[t,
t^{-1}; \sigma]$ be the skew-Laurent ring.  Our convention is to
write coefficients on the left; so either ring satisfies the
relations $ts = \sigma(s) t$ for all $s \in S$.  Let $X = \spec S$;
we use the same name $\sigma: X \to X$ for the scheme automorphism
of $X$ corresponding to the algebra automorphism $\sigma$.
\end{example}

Notice the similarity between the following characterization of
homogeneous primes of such rings and Lemma~\ref{thcr-primes-lem}.
\begin{lemma}
\label{skew-primes-lem} Consider Example~\ref{skew-ex}.
\begin{enumerate}
\item Suppose that $Z \subseteq X$ is a reduced closed subscheme such
that $\sigma(Z) = Z$ and $\sigma$ acts as a single cycle on the
irreducible components of $Z$.  If $I$ is the (radical) ideal of $S$
corresponding to $Z$, then $P = \bigoplus_{n \in \mb{Z}} I t^n$ is a
homogenous prime ideal of $T$; moreover, all homogeneous prime
ideals of $T$ have this form.

\item The homogeneous primes of $U$ are exactly those of the form $P
\cap U = \bigoplus_{n \geq 0} I t^n$ where $P$ is as in part (1),
together with those of the form  $Q = J \oplus St \oplus St^2 \oplus
\dots$ where $J$ is any prime ideal of $S$.
\end{enumerate}
\end{lemma}
\begin{proof}
This result is well-known.  Cf. \cite[Lemma 2.3, 2.4]{Jo} for the
claim about $T$. The claim for $U$ follows since $T$ is the
localization of $U$ at the multiplicative system $\{ 1, t, t^2,
\dots \}$.
\end{proof}

Primitivity of prime rings of the form $T$ is characterized in
\cite{Jo}, and for prime rings of the form $U$ in \cite{LM}; we
review the criteria in Section~\ref{summary-sec}. However, it is not
clear from these characterizations (which go by constructing
faithful simple modules in the primitive case) when the
Dixmier-Moeglin equivalence holds for these rings.   Our methods
will allow us to address this question for all of the examples we
have given above, through a study of the geometry of $(X, \sigma)$.
The following definitions will be useful for the geometric point of
view.

\begin{definition}
\label{sigma-irr-def} Let $X$ be a scheme of finite type over $k$,
with $k$-automorphism $\sigma: X \to X$.  A closed subset $Z$ of $X$
is called \emph{$\sigma$-invariant} if $\sigma(Z) = Z$ and
\emph{$\sigma$-periodic} if $\sigma^n(Z) = Z$ for some $n \geq 1$. A
$\sigma$-invariant closed subset $Z$ is called
\emph{$\sigma$-irreducible} if there does not exist a decomposition
$Z = Z_1 \cup Z_2$ with $Z_1, Z_2 \subsetneq Z$ closed
$\sigma$-invariant subsets. If $X$ is itself $\sigma$-irreducible, a
\emph{maximal $\sigma$-irreducible} subset $Z$ of $X$ is a maximal
element of the set of all \emph{proper} $\sigma$-irreducible subsets
of $X$ under inclusion.

For use in Section~\ref{base-ext-sec}, we note that the notions
defined here make sense for an arbitrary ground field $k$.
\end{definition}
Notice that a $\sigma$-invariant subset $Z \subseteq X$ is
$\sigma$-irreducible if and only if $\sigma$ permutes the finitely
many irreducible components of $Z$ in a single cycle.

\begin{definition}
\label{geom-ord-def} Let $X$ be a scheme of finite type over $k$,
with $k$-automorphism $\sigma: X \to X$.  We say that the pair $(X,
\sigma)$ is \emph{\ord} if it satisfies the following property: for
every $\sigma$-irreducible closed subset $Z \subseteq X$, $(Z,
\sigma \vert_Z)$ has either finitely many or uncountably many
maximal $\sigma$-irreducible closed subsets.
\end{definition}

In the main result of this section, we now show that for all of the
examples $A$ above, the property of $(X, \sigma)$ being \ord\ is the
geometric equivalent of the property of $\hspec A$ being \hord. We
also verify that all of these examples satisfy the other hypotheses
of Proposition~\ref{hord-DM-prop}, so that to prove the
DM-equivalence for them it will be enough to study the geometry of
$(X, \sigma)$.
\begin{proposition}
\label{geom-primes-prop} Let $A$ be any one of the following
$k$-algebras:
\begin{enumerate}
\item A twisted homogeneous coordinate ring $A = B = B(X, \mc{L}, \sigma)$
as in Example~\ref{thcr-ex};
\item A skew-Laurent ring $A =T = S[t, t^{-1}; \sigma]$ as in Example~\ref{skew-ex} ; or
\item A skew polynomial ring $A =U = S[t; \sigma]$ as in
Example~\ref{skew-ex}.
\end{enumerate}
In each case, we have the corresponding geometric data $(X,
\sigma)$; we remind the reader that in cases $(2)$ and $(3)$, $X =
\spec S$.  Then $(X, \sigma)$ is \ord\ if and only if $\hspec A$ is
\hord, and in this case $\spec A$ is \ord\ and $A$ satisfies the
Dixmier-Moeglin equivalence.
\end{proposition}
\begin{proof}
In cases (1) and (2), Lemmas~\ref{thcr-primes-lem} and
\ref{skew-primes-lem} show that there is a one-to-one
inclusion-reversing correspondence between the homogeneous primes of
$A$ and $\sigma$-irreducible subschemes of $X$.  Then it follows
directly from definitions that $(X, \sigma)$ is \ord\ if and only if
$\hspec A$ is \hord.   Now consider case (3), where $A = U = S[t;
\sigma] \subseteq T = S[t, t^{-1}; \sigma]$. Lemma
\ref{skew-primes-lem} shows that the topological space $\hspec U$ is
the disjoint union of a subspace $W_1 \cong \hspec T$ and a subspace
$W_2 \cong \spec S$.   Moreover, for a prime $P \in W_2$, the primes
of height one over $P$ lie again in $W_2$; while for a prime $P$ in
$W_1$, the primes of height one over $P$ are the primes of height
one over $P$ in $W_1$ together with finitely many primes from $W_2$.
Since $\spec S$ is automatically \ord\ by Lemma~\ref{PI-lem}, it
follows that $\hspec U$ is \hord\ if and only if $\hspec T$ is
\hord, which is if and only if $(X, \sigma)$ is \ord\ as we have
already seen.

To finish the proof, we will apply Proposition~\ref{hord-DM-prop}.
It thus remains only to verify the hypotheses of that proposition,
which we do case by case.

(1) Let $A = B = B(X, \mc{L}, \sigma)$ as in Example~\ref{thcr-ex}.
Since $\mc{L}$ is $\sigma$-ample, $B$ is noetherian \cite[Theorem
1.2]{Ke1}. Since $B$ is noetherian and $\mb{N}$-graded with $B_0 =
k$ and $\dim_k B_n < \infty$ for all $n \geq 0$, it follows that the
right $B$-module $k = B/B_{\geq 1}$ has a minimal free resolution by
free modules of finite rank, say
\[
\dots \to B^{n_1} \to B^{n_0} \to B \to k \to 0,
\]
where $n_0$ is the minimal number of generators of $B$ as a graded
$k$-algebra and $n_1$ is the minimum number of relations; see, for
example, \cite[p. 42-43]{ATV1}.  Thus $B$ is finitely presented. It
is clear that $B$ has DCC on homogeneous prime ideals, using
Lemma~\ref{thcr-primes-lem} and the fact that $X$ has finite
dimension.

Now consider a homogeneous prime factor ring $B/P$ of $B$, where $P$
corresponds via Lemma~\ref{thcr-primes-lem} to the
$\sigma$-irreducible closed reduced subscheme $Z \subseteq X$ with
ideal sheaf $\mc{I}$.   Then there is a natural homomorphism $\phi:
B = B(X, \mc{L}, \sigma) \to B' = B(Z, \mc{L} \vert_Z, \sigma
\vert_Z)$ given by restriction of sections, with kernel $P$.  It is
easy to see that $\phi$ is surjective in large degree, by the
definition of $\sigma$-ampleness of $\mc{L}$.  Thus the graded
quotient ring of $B/P$ is isomorphic to the graded quotient ring of
$B'$.  But $Q_{\rm gr}(B') \cong R[t, t^{-1}; \sigma]$, where $R$ is
the product of the function fields of the irreducible components of
$Z$, and $\sigma$ is the induced automorphism of $R$ (see, for
example, the proof of \cite[Proposition 3.5]{RZ}.)  Thus, every
homogeneous prime factor ring of $B$ is birationally commutative, so
certainly birationally PI.

(2), (3)  Since $S$ is noetherian and finitely generated, both  $T$
and $U$ are noetherian, finitely generated $k$-algebras.  It is
immediate from Lemma~\ref{skew-primes-lem} that $T$ and $U$ have DCC
on homogeneous primes.  The fact that all homogeneous prime factor
rings of $T$ and $U$ are birationally commutative is also clear from
Lemma~\ref{skew-primes-lem}.
\end{proof}

We remark that the same idea as in the previous theorem applies
equally well to many other birationally commutative algebras, for
example the \naive\ blowups studied in \cite{RS1}, or the geometric
idealizer rings of \cite{Si}; in each case, the given ring $A$ is a
subring of a twisted homogeneous coordinate ring $B$ for which one
can show that $\hspec A$ and $\hspec B$ are isomorphic.  We have
chosen to focus here only on a few representative examples.

\section{Orbital characterization of ordinary automorphisms}

\label{orbital-char-sec}

Hypothesis~\ref{k-hyp2} remains in force throughout this section.
Our eventual goal is to apply Proposition~\ref{geom-primes-prop} to
show that many twisted homogeneous coordinate rings satisfy the
DM-equivalence. In order to do so, we need first to understand
better the property of $(X, \sigma)$ being \ord. Clearly, a
condition on the cardinality of maximal $\sigma$-irreducible subsets
is not very natural from a geometric standpoint.  In this section we
show that the \ord\ condition can be restated in a much nicer way,
in terms of the orbits of the $\sigma$-action on $X$. This new
characterization will be much more readily amenable to further study
using geometric methods.

We first record for reference a standard result.
\begin{lemma}
\label{amitsur-trick-lem} Let $X$ be an integral scheme of finite
type over $k$.  Then $X$ cannot be written as a union of countably
many proper closed subsets.
\end{lemma}
\begin{proof}
This is a well-known consequence of the uncountability of $k$.  A
proof in case $X$ is affine may be found in \cite[Corollary
3.4]{Bell}, and the general case follows immediately from this.
\end{proof}

Suppose that $X$ is an integral scheme for which $(X, \sigma)$ is
\ord.  If $X$ has finitely many maximal $\sigma$-irreducible
subsets, then the picture is quite clear:  the union of these
finitely many sets is a $\sigma$-invariant closed set $Z$ containing
all proper $\sigma$-invariant closed sets of $X$, and for every
point $x \not \in Z$, $x$ lies on a dense $\sigma$-orbit.  On the
other hand, if $X$ has uncountably many maximal $\sigma$-irreducible
subsets, it is not immediately clear how to picture $\sigma$.  For
example, can $X$ also have some dense $\sigma$-orbits?  In the
following theorem, we show that in the second case the orbital
picture is also clear: surprisingly, $X$ must be completely covered
by $\sigma$-invariant codimension-1 subschemes arising as the fibers
of a $\sigma$-invariant rational function.  In particular, $X$ has
no dense $\sigma$-orbits.  This proves part of
Theorem~\ref{main-geom-thm} from the introduction.

\begin{theorem}
\label{unc-irr-thm} Assume $k$ is uncountable and algebraically
closed. Let $X$ be an integral scheme of finite type over $k$, with
$\sigma: X \to X$ a $k$-automorphism and $\sigma: k(X) \to k(X)$ the
induced automorphism of the field of rational functions. Then the
following are equivalent:
\begin{enumerate}
\item There exists a non-constant $\sigma$-invariant rational function $f \in
k(X)$.
\item There exists a non-constant $\sigma^m$-invariant rational
function $f \in k(X)$ for some $m \geq 1$.
\item $X$ has no dense $\sigma$-orbits.
\item $X$ has uncountably many maximal $\sigma$-irreducible closed
subsets.
\item $X$ has uncountably many $\sigma$-irreducible closed subsets of codimension-1.
\end{enumerate}
\end{theorem}
\begin{proof}
$(1) \implies (2)$ is immediate.

$(2) \implies (3)$:  Suppose that $\sigma^m(f) = f \circ \sigma^m =
f$ for some $m \geq 1$, and let $U$ be the largest open set of $X$
on which the rational function $f$ is defined.
 For each $a \in k$, $f^{-1}(a)$ is a
$\sigma^m$-invariant closed subset of $U$, which is of codimension-1
because $f$ is nonconstant. The set $U$ is covered by such fibers of
$f$; moreover, $X \setminus U$ is also a $\sigma^m$-invariant
codimension-1 closed set.   So $X$ is covered by uncountably many
$\sigma$-periodic codimension-1 closed sets, and certainly then
there are no dense $\sigma$-orbits.

This also shows that $(2) \implies (5)$, and $(5) \implies (4)$ is
obvious.

$(3) \implies (4)$:  This is a consequence of
Lemma~\ref{amitsur-trick-lem}.

$(4) \implies (1)$: Let $\Omega$ be an uncountable set of maximal
$\sigma$-irreducible closed subsets of $X$.  Pick any open affine
subset $U = \spec R$ of $X$, where $R$ is a domain finitely
generated over $k$.  Let $K = k(X)$, the fraction field of $R$. For
each $n \in \ZZ$, put $V_n = \sigma^n(U)$ and $Y_n = X \setminus
V_n$. We claim that $\Omega$ contains an uncountable subset
$\Omega'$ so that for any $Z \in \Omega'$, we have that $Z
\not\subseteq \bigcup_{n\in\ZZ} Y_n$.  To see this, note that any
finite union of some of the  $Y_n$ can contain at most finitely many
of the sets in $\Omega$, as the closure of the union of any infinite
collection of sets in $\Omega$ is $\sigma$-invariant and so equal to
$X$.

Without loss of generality, we may replace $\Omega$ by $\Omega'$.
Then as $\bigcup_{n \in \ZZ} Y_n$ is $\sigma$-invariant, we see that
any irreducible component of any $Z \in \Omega$ intersects $V =
\bigcap_{n \in \ZZ} V_n$ nontrivially.

Fix once and for all a countable $k$-basis $r_1, r_2, r_3, \dots$
for $R$. Let $T$ be the union of the closed points in $Z \cap V$ as
$Z$ ranges over all sets in $\Omega$.  For each $x \in T$, let $H =
H_x$ be the closure in $X$ of $\mc{O}_x = \{ \sigma^n(x) | n \in
\mb{Z} \}$. Since there is some $Z \in \Omega$ so that $H_x \cap U
\subseteq Z \cap U$  is a proper nonempty closed subset of $U$,
there is some nonzero regular function $g \in R$ such that $g$
vanishes along $H_x$. Write $g = \sum_{i = 1}^m a_i r_i$. We may
choose such a $g$ with minimal possible $m$; then $g$ is uniquely
determined by $x$ up to scalar and we put $g = g_x$ and $m = m_x$.
For each $p \geq 1$, let $T_p = \{x \in T | m_x = p \}$. Note that
$T_p$ is $\sigma$-invariant.
 For any $Z
\in \Omega$, note also that
\[
Z = \bigg(\bigcup_{p \geq 1} (Z \cap T_p)\bigg) \bigcup
\bigg(\bigcup_{n \in \mb{Z}} (Z \cap Y_n)\bigg).
\]
Since $\sigma$ permutes the components of $Z$ in a cycle and each
subset $Z \cap T_p$ is $\sigma$-invariant, we see by
Lemma~\ref{amitsur-trick-lem} that $Z \cap T_p$ is dense in $Z$ for
some $p \geq 1$. As $\Omega$ is uncountable, we can find some fixed
$p \geq 1$ such that $Z \cap T_p$ is dense in $Z$ for uncountably
many $Z \in \Omega$. Letting $S = T_p$ for this $p$, we see that $S$
has the following properties: $S$ is uncountable; $S \subseteq V$;
$S$ is dense in $X$ (since its closure contains infinitely many
maximal $\sigma$-irreducibles); and for each $x \in S$, $m_x = p$ .

For any $n \in \mb{Z}$ we write $r_i^{\sigma^n}$ for the rational
function $\sigma^n(r_i) = r_i \circ \sigma^n \in K$.  For each $n
\in \mb{Z}$ let $v_n = (r_1^{\sigma^n}, r_2 ^{\sigma^n}, \dots,
r_p^{\sigma^n}) \in K^p$, and let $W = \operatorname{span}_K \{ v_n
| n \in \mb{Z} \}$. Note that for any $x \in S$, we can evaluate
$v_n$ at $x$, obtaining $v_n(x)= (r_1(\sigma^n(x)),
r_2(\sigma^n(x)), \dots, r_p(\sigma^n(x))) \in k^p$.  Let $W(x) =
\operatorname{span}_k \{ v_n(x) | n \in \mb{Z} \}$. Now by
construction, $g_x = \sum_{i = 1}^p a_i r_i$ is the unique (up to
scalar) nonzero $k$-linear combination of $r_1, \dots, r_p$ which
vanishes at $\sigma^n(x)$ for all $n \in \mb{Z}$.  Thus $(a_1,
\dots, a_p)$ spans the orthogonal complement to $W(x)$ in $k^p$, and
so $\dim_k W(x) = p-1$.  In fact, then we can find a fixed sequence
of integers $n_1, n_2, \dots, n_{p-1}$ and a dense subset $S'
\subseteq S$ such that $v_{n_1}(x), \dots, v_{n_{p-1}}(x)$ is a
$k$-basis for $W(x)$, for all $x \in S'$ (where here we use
Lemma~\ref{amitsur-trick-lem} again.)

We claim now that $\dim_K W = p-1$.  First, suppose that $v_{n_1},
\dots, v_{n_{p-1}}$ are $K$-linearly dependent, say $\sum_{j =
1}^{p-1} b_j v_{n_j} = 0$ with $b_j \in K$ not all zero.  Since $S'$
is dense, we can choose $x \in S'$ such that $b_j$ is defined at $x$
for all $1 \leq j \leq p-1$ and $b_j(x) \neq 0$ for some $j$.  Then
evaluating at $x$ we have $\sum_{j = 1}^{p-1} b_j(x) v_{n_j}(x) = 0$
and it follows that $v_{n_1}(x), \dots, v_{n_{p-1}}(x)$ are
$k$-linearly dependent, a contradiction.  Thus $v_{n_1}, \dots,
v_{n_{p-1}}$ are linearly independent over $K$.   On the other hand,
suppose that $\dim_K W = p$, so we may choose $n_p$ such that
$v_{n_1}, \dots, v_{n_p}$ is a $K$-basis for $W$. If we consider the
$p \times p$ matrix $M = (M_{ij}) = (r_j^{\sigma^{n_i}}) \in
M_p(K)$, then $0 \neq \det M \in K$. However, we know from the
previous paragraph that for any $x \in S$, $v_{n_1}(x), \dots,
v_{n_p}(x)$ are linearly dependent over $k$ and so $(\det M)(x) =
0$. Since $S$ is dense in $X$ and $\det M$ is a nonzero rational
function, this is a contradiction.  So $\dim_K W = p-1$ as claimed.

Pick $F = (f_1, f_2, \dots, f_p) \in K^p$ which spans the orthogonal
complement to $W$ in $K^p$.   The vector $F^{\sigma} =
(f_1^{\sigma}, f_2^{\sigma}, \dots, f_p^{\sigma})$ is in the
orthogonal complement to $W^{\sigma}$, but by construction
$W^{\sigma} = W$.  Thus $F^{\sigma} = \lambda F$ for some $0 \neq
\lambda \in K$ and we conclude that $(f_i/f_j)^{\sigma} = f_i/f_j$
for any $i,j$ such that $f_j \neq 0$.  Finally, we claim that if
$f_j \neq 0$, then there is some $i$ so that $b_i = f_i/f_j \not\in
k$. For if $b_i \in k$ for all $i$, then  we would have a
$k$-dependency $b_1 r_1 + b_2 r_2 + \dots + b_p r_p = 0$,
contradicting the initial choice of the $r_i$ as a $k$-basis of $R$.
So picking $j$ such that $f_j \neq 0$ and $i$ such that $h =
f_i/f_j$ is not in $k$, we obtain a non-constant $\sigma$-invariant
rational function $h$ on $X$.
\end{proof}

With the preceding result in hand, our alternative characterization
of the \ord\ property, which we call simply \emph{ordinary}, is very
easy to prove.  Here are the relevant definitions.
\begin{definition}
\label{gdo-def} Let $X$ be a scheme and let $\sigma: X \to X$ be an
automorphism. For any $x \in X$, let $\mc{O}_x = \{ \sigma^n(x) | n
\in \mb{Z} \}$ be the $\sigma$-orbit of $X$.  We say that $(X,
\sigma)$ has \emph{good dense orbits} if the set
\[ U = \{ x \in X |
\mc{O}_x\ \text{is Zariski dense in}\ X \}
\]
is a (Zariski) open subset of $X$.  We say that $(X, \sigma)$ is
\emph{ordinary} if for all $\sigma$-irreducible closed subsets $Z
\subseteq X$, $(Z, \sigma \vert_Z)$ has good dense orbits.
\end{definition}

\begin{corollary}
\label{gdo-ord-cor} Let $X$ be a scheme of finite type over $k$ with
$k$-automorphism $\sigma: X \to X$.  Then $(X, \sigma)$ is \ord\ if
and only if $(X, \sigma)$ is ordinary.
\end{corollary}
\begin{proof}
Clearly it suffices to assume that $X$ is reduced and
$\sigma$-irreducible, and to prove that $X$ has good dense orbits if
and only if $X$ has either finitely many or uncountably many maximal
$\sigma$-irreducible subsets.  Replacing $\sigma$ by some power
which fixes each component of $X$, we see that we may also reduce to
the case that $X$ is integral.

Now
by Proposition~\ref{unc-irr-thm}, $X$ has no dense $\sigma$-orbits
 if and only if $X$ has uncountably many
maximal $\sigma$-irreducible subsets.  On the other hand, it is easy
to see that the points lying on a dense $\sigma$-orbit form a
nonempty open set $U$ if and only if $X$ has only finitely many
maximal $\sigma$-irreducible subsets, say $Y_1, \dots, Y_n$, whose
union is $Y = X \setminus U$.
\end{proof}

\section{Base extension}

\label{base-ext-sec}

In this section, we study how the maximal $\sigma$-irreducible
subsets of a pair $(X, \sigma)$ are related to those of $(X',
\sigma)$, where $X'$ is some base field extension of $X$.  We then
offer two important applications.  First, we will show that the
property of an automorphism having good dense orbits, as in
Definition~\ref{gdo-def}, is invariant under extension of the base
field. This will enable us in later sections, by the usual Lefschetz
principle, to change the base field from an uncountable
algebraically closed field $k$ of characteristic $0$ to the case $k
= \mb{C}$, so that results from complex algebraic geometry can be
applied.   The second application is more subtle.  Suppose that $(X,
\sigma)$ fails to have good dense orbits, for some integral scheme
$X$. We already know from Theorem~\ref{unc-irr-thm} that in this
case, $X$ must have a countably infinite number of maximal
$\sigma$-irreducible closed subsets.  In
Theorem~\ref{inv-divisor-thm} below, we show the surprising result
that only finitely many of those $\sigma$-irreducible closed subsets
can be of codimension-$1$ in $X$.  This will a key ingredient in our
later analysis of the case where $X$ is a surface. The idea of the
proof of Theorem~\ref{inv-divisor-thm} is to reduce to the case of a
base field $F$ which is finitely generated over its prime subfield,
and then use the fact that the divisor class group of a variety over
such a field $F$ is a finitely generated group.  Thus, we will need
to prove our general base field extension results for arbitrary
fields.

\begin{notation}
\label{extension-not}  In all of the results in this section, we
assume the following setup and notation.  Suppose that $Y$ is an
integral scheme of finite type over an arbitrary  field $F$.  Let $F \subseteq
E$ be a field extension and set $X = Y \times_{\spec F} \spec E$.
Let $\sigma: Y \to Y$ be an automorphism of $Y$ as an $F$-scheme,
and let $\sigma: X \to X$ also denote the induced $E$-automorphism
of $X$.  Let $\pi: X \to Y$ be the projection morphism.
\end{notation}

We first single out a few simple observations about this setup.
\begin{lemma}
\label{ac-ext-lem} Assume Notation~\ref{extension-not}.  If $F$ is
algebraically closed, then $X$ is also integral.
\end{lemma}
\begin{proof}
This amounts to the ring theoretic statement that if $R$ is a
commutative $F$-algebra which is a domain, then $R \otimes_F E$ is
also a domain.  This fact is well-known; for instance see
\cite[Proposition 17.2]{Ber} for a more general result.
\end{proof}

\begin{lemma}
\label{pullback-dim-lem}  Assume Notation~\ref{extension-not}.  For
every irreducible closed subset $Z \subseteq Y$, every irreducible
component of $\pi^{-1}(Z)$ has the same dimension as $Z$.
\end{lemma}
\begin{proof}
This result is surely well-known, but for lack of an exact
reference, we sketch the proof. One may reduce immediately to the
affine case; thus one needs to prove the following fact about rings:
given a finitely generated commutative $F$-algebra $R$ which is a
domain with $\dim R = d$, then every minimal prime $P$ of $S = R
\otimes_F E$ has $\dim (S/P) = d$. By Noether normalization, we may
choose a polynomial ring $F[y_1, y_2, \dots y_d] \subseteq R$ over
which $R$ is finite; then $S$ is finite over $E[y_1, \dots y_d]$ and
so $\dim S = d$. Now if $E$ is purely transcendental over $F$, then
$S$ is a domain and the result follows.  Thus we can reduce to the
case that $E$ is algebraic over $F$.  In this case, $S$ is both flat
(in fact free) over $R$ and integral over $R$.  Given any minimal
prime $P$ of $S$, $P \cap R = 0$ by the going-down theorem
\cite[Theorem 9.5]{Ma}. Then $\dim S/P
> \dim R$ by the going-up theorem \cite[Theorem 9.4]{Ma}, so $\dim
S/P = \dim R = d$.
\end{proof}

Now we prove our main lemma about base extension.  The idea of the
proof of part (1) is similar to some other results in the
literature; for example, cf. \cite[Theorem 2.1]{Ir}.
\begin{lemma}
\label{base-ext-lem} Assume Notation~\ref{extension-not}, and in
addition assume that $X$ is integral.  Suppose that for every
$\sigma$-invariant rational function $f \in E(X)$, $f$ is algebraic
over $E$. Then the following hold:
\begin{enumerate}
\item  If $Z \subsetneq X$ is a proper $\sigma$-invariant closed subset of $X$, then
$\overline{\pi(Z)} \subsetneq Y$ is a proper $\sigma$-invariant
closed subset of $Y$.

\item The rule $Z \mapsto
\overline{\pi(Z)}$ defines a surjective function
\[
\phi: \{ \text{maximal}\ \sigma-\text{irreducible subsets of}\ X \}
\to \{ \text{maximal}\ \sigma-\text{irreducible subsets of}\ Y \}
\]
which is finite-to-$1$ and dimension-preserving.
\end{enumerate}
\end{lemma}
\begin{proof}
(1) Choose any open affine set $U' \subseteq Y$ and let $U
=\pi^{-1}(U')$.  If $U' = \spec B$, then $U = \spec A$ where $A = B
\otimes_F E$. Let $\mc{I}$ be the ideal sheaf of the reduced
subscheme $Z$.  Let $I = \mc{I}(U)$, which is a nonzero ideal of
$A$. The only nontrivial thing we need to show is that
$\overline{\pi(Z)}$ is proper in $Y$, and for this it will suffice
to show that $I \cap B \neq 0$.

Pick a nonzero element
\[
x = \sum_{i=1}^d b_i \otimes \alpha_i\in I
\]
with $b_i\in B$, $\alpha_i\in E$.  We choose such an element with
the minimal $d$ for all possible choices of open affine subset $U'$.
If $d=1$, then $x(1\otimes \alpha_1^{-1})$ is a nonzero element of
$I \cap B$ and we are done, so we may assume that $d>1$.  We will
deduce a contradiction.   By minimality, the $\alpha_i$ are linearly
independent over $F$. Since $\mc{I}$ is $\sigma$-invariant, we have
\[
\sigma(x) = \sum_{i=1}^d \sigma(b_i)\otimes \alpha_i\in
\mc{I}(\sigma^{-1}(U))
\]
and thus working on the open affine set $U \cap \sigma^{-1}(U) =
\pi^{-1}(U' \cap \sigma^{-1}(U'))$ we see that
\[
b_d\sigma(x)-\sigma(b_d)x = \sum_{i=1}^{d-1}
(b_d\sigma(b_i)-\sigma(b_d)b_i)\otimes \alpha_i \in \mc{I}(U \cap
\sigma^{-1}(U)).
\]
By minimality of $d$, this element must be zero. Hence, by linear
independence of the $\alpha_i$, we have
\[
b_i/b_d = \sigma(b_i/b_d)
\]
as elements of $F(Y) = Q(B)$.  Thus $b_i/b_d \in E(X) = Q(A)$ is
algebraic over $E$ for $1 \leq i \leq d$ by hypothesis.

Since $A$ is a domain, the localization $Q(B) \otimes_F E$ is also a
domain, and it contains the elements $b_i/b_d$ for all $i$. Since
these elements are algebraic over $E$, $L = E[b_1/b_d, b_2/b_d,
\dots, b_{d-1}/b_d]$ is a finite-dimensional $E$-subspace of a
domain, and hence a field. Clearly there is $m >0 $ such that $b_d^m
\lambda \in B \otimes_F E = A$ for all $\lambda\in L$.  Write
 $y = x (b_d \otimes 1)^{-1} =  \sum_{i=1}^d b_i/b_d \otimes \alpha_i
\in L$.  Since $y \neq 0$ it has an inverse $y^{-1} = \sum_j
\gamma_j \otimes \beta_j \in L$. Then
\[
0 \neq x' = y^{-1} (b_d^m \otimes 1) =  \sum_j (b_d^m
\gamma_j)\otimes \beta_j \in A
\]
and
\[
0 \neq xx' = (b_d \otimes 1) y y^{-1} (b_d^m \otimes 1) =
b_d^{m+1}\otimes 1\in I,
\]
contradicting the assumption that $d$ is greater than $1$. Thus the case $d> 1$ cannot occur.

(2)
 Given a $\sigma$-irreducible closed subset $W \subseteq Y$,
$\pi^{-1}(W)$ is clearly $\sigma$-periodic, so is a finite union of
$\sigma$-irreducibles of $X$.  Moreover, all components of
$\pi^{-1}(W)$ have the same dimension as $W$, by
Lemma~\ref{pullback-dim-lem}.   Conversely, if $Z \subseteq X$ is a
$\sigma$-irreducible closed subset, it is clear that $Z' =
\overline{\pi(Z)}$ is $\sigma$-irreducible in $Y$.

Now if $Z$ as above is a \emph{maximal} $\sigma$-irreducible, then
$Z'$ is a proper subset of $Y$ by part (1), so $\pi^{-1}(Z')$ is a
proper subset of $X$ and by maximality $Z$ is a union of components
of $\pi^{-1}(Z')$.  In particular, in this case $\dim Z' = \dim Z$.
Also, $Z'$ is a maximal $\sigma$-irreducible of $Y$: otherwise $Z'
\subsetneq W \subsetneq Y$ for some $\sigma$-irreducible $W$, and
then $Z \subsetneq \pi^{-1}(W) \subsetneq X$ by dimension
considerations, a contradiction.  Thus there is a well-defined
function $\phi$ as in the statement of the lemma, and it preserves
dimensions. Conversely, if $W$ is a maximal $\sigma$-irreducible of
$Y$, then picking any maximal $\sigma$-irreducible $Z$ of $X$ such
that $Z$ contains a component of $\pi^{-1}(W)$, we have $W \subseteq
Z' = \overline{\pi(Z)}$ and thus $Z' = W$, so $\phi$ is surjective.
Moreover, we must have $\dim Z = \dim Z' = \dim W = \dim
\pi^{-1}(W)$, so $Z$ is a union of components of $\pi^{-1}(W)$.  It
follows that $\phi$ is finite-to-one.
\end{proof}

Next, we immediately apply the preceding lemma, to show that the
property of having good dense orbits is invariant under extension of
the base field.
\begin{lemma}
\label{gdo-basechange-lem} Assume Notation~\ref{extension-not}, and
suppose in addition that $F = k$ and $E = \ell$ are algebraically
closed uncountable fields.   Then $(X, \sigma)$ has good dense
orbits if and only if $(Y, \sigma)$ has good dense orbits.
\end{lemma}
\begin{proof}
Since $k$ is algebraically closed, and $Y$ is integral, the extended
scheme $X = Y \times_{\spec k} \spec \ell$ is also integral by
Lemma~\ref{ac-ext-lem}.  Suppose first that $(X, \sigma)$ has a
dense orbit. Then by Theorem~\ref{unc-irr-thm}, we see that $X$ has
no non-constant $\sigma$-invariant rational functions, and so the
hypotheses of Lemma~\ref{base-ext-lem} hold. Thus by
Lemma~\ref{base-ext-lem}(2), the number of maximal
$\sigma$-irreducible closed subsets of $Y$ is either finite or
uncountable if and only if the same holds for $X$.  Then by the
proof of Corollary~\ref{gdo-ord-cor}, $(X, \sigma)$ has good dense
orbits if and only if $(Y, \sigma)$ does.

If, on the other hand, $(X, \sigma)$ has no dense orbits, then it
will suffice to prove that $(Y, \sigma)$ also has no dense orbits.
Note that given a closed point $y \in Y$, then $x = \pi^{-1}(y)$ is
also a closed point of $X$, since $k$ is algebraically closed.  We
claim that if $y$ lies on a dense $\sigma$-orbit, then so does $x$.
In fact, a more general statement holds: given a set of closed
points $S  \subseteq Y$, if $S$ is dense in $Y$ then $\pi^{-1}(S)$
is dense in $X$.  To see this, one may reduce to the affine case,
say $Y = \spec B$ and $X = \spec A$ with $A = B \otimes_k \ell$.  If
for $y \in S$ we let $I_y$ be the corresponding maximal ideal of
$B$, then  by assumption $\bigcap I_y = 0$; thus $\bigcap (I_y
\otimes \ell) = 0$ in $A$ and the claim follows.
\end{proof}

We now move toward our second application of the base field
extension lemma.  We first recall in the next lemma some needed
background results, which are presumably well-known.  We work in the
rest of this section with Weil divisors on a normal variety $X$; see
\cite[Section II.6]{Ha} for more background. Recall that the divisor
class group of $X$, $\cl X$, is the group of Weil divisors modulo
linear equivalence (which might be different from the Picard group
$\pic X$ if $X$ is not locally factorial.)

\begin{lemma}
\label{pic-fg-lem} Let $X$ be a normal quasi-projective variety over
a field $F$.
\begin{enumerate}
\item If $F$ is a finitely generated extension of its prime
subfield, then $\cl X$ is a finitely generated abelian group.
\item If $G = \Gamma(X, \mc{O}_X)^*$ is the group of units of the
global regular functions on $X$, then $G/(\overline{F}^* \cap G)$ is
a finitely generated abelian group.
\end{enumerate}
\end{lemma}
\begin{proof}
(1) By definition, $X$ can be embedded as an open subset of a
projective variety $Y$; replacing $Y$ by its normalization if
necessary (which does not affect the open set $X$), we may assume
that $Y$ is normal and projective.  Now $\cl X$ is a surjective
image of $\cl Y$ by \cite[Proposition II.6.5]{Ha}, so we need only
show that $\cl Y$ is finitely generated.  This is immediate from the
version of the Mordell-Weil-N\'eron-Severi theorem in
\cite[Corollary 6.6.2]{La2}.

(2) As in part (1), we may embed $X$ as an open subset of a normal
projective $Y$. Then $Y \setminus X$ has finitely many irreducible
components, some of which are of codimension 1, say $Z_1, \dots,
Z_n$.  We thus have a homomorphism from G to the additive group
$\mathbb{Z}^n$, given by $f \mapsto (v_{Z_1}(f), \dots,
v_{Z_n}(f))$, where $v_{Z_i}$ is the discrete valuation associated
to the divisor $Z_i$. The kernel of this homomorphism is $\Gamma(Y,
\mc{O}_Y)^*$, which, since $Y$ is projective, is contained in
$\overline{F}^*$.  The result follows.
\end{proof}

We now give the result promised at the beginning of this section,
which shows in particular that if $(X, \sigma)$  fails to have good
dense orbits, all but a finite number of the countably many maximal
$\sigma$-irreducible subsets must be of codimension-2 or smaller.
This completes the proof of Theorem~\ref{main-geom-thm} from the
introduction.

\begin{theorem}
\label{inv-divisor-thm} Let $k$ be uncountable and algebraically
closed, and let $X$ be a quasi-projective integral scheme over $k$
with $k$-automorphism $\sigma: X \to X$. If $X$ has infinitely many
$\sigma$-irreducible codimension-$1$ closed subsets, then there
exists a non-constant $\sigma$-invariant rational function $f \in
k(X)$. Consequently, $\sigma$ does not have a dense orbit, and there
are uncountably many $\sigma$-irreducible codimension-1 closed
subsets.
\end{theorem}
\begin{proof}
Let $X$ have infinitely many $\sigma$-irreducible codimension-$1$
closed subsets.  Since the hypotheses pass to the normalization
$\wt{X}$ of $X$, we may assume that $X$ is normal. It is standard
that we can pick a finitely generated extension $F$ of the prime
subfield of $k$ so that $X$ and $\sigma$ are defined over $F$. In
other words, we can find a quasi-projective $F$-variety $Y$ such
that $X = Y \times_{\spec F} \spec k$ and an automorphism $\sigma: Y
\to Y$ which induces the given $\sigma: X \to X$.  Thus we are in
the situation of Notation~\ref{extension-not}, with $E = k$.  As
there, let $\pi: X \to Y$ be the projection.  Note that $Y$ is also
normal.

We are certainly done if there exists a nonconstant $f \in k(X)$
with $\sigma(f) = f$, so we may assume that this is not the case,
and then the hypotheses of Lemma~\ref{base-ext-lem} hold. Note that
as a special case of the correspondence in part (2) of that lemma,
the rule $Z \to \overline{\pi(Z)}$ defines a finite-to-one mapping
from the set of $\sigma$-irreducible codimension-$1$ closed subsets
of $X$ to the set of $\sigma$-irreducible codimension-$1$ closed
subsets of $Y$.  In particular, we conclude that $Y$ also has
infinitely many $\sigma$-irreducible codimension-$1$ closed subsets.

Now by Lemma~\ref{pic-fg-lem}, $\cl Y$ and  $H = \Gamma(Y,
\mc{O}_Y)^*/(\Gamma(X, \mc{O}_Y)^* \cap \overline{F}^*)$ are
finitely generated abelian groups.  Moreover, it easy to see that
$H$ is torsionfree; let $d$ be its torsionfree rank. We may find a
sequence of distinct irreducible codimension-one closed subsets
$Y_1, Y_2, \ldots $ which are $\sigma$-periodic. Let $[Y_i]$ be the
divisor class in $\cl Y$ determined by $Y_i$. Then there is some $n$
such that in $\cl Y$,
$$
[Y_m]\in \sum_{i=1}^n \mathbb{Z}[Y_i]
$$
for all $m$.
  For each $m>n$, there is some rational
function $f_m$ on $Y$ with corresponding principal divisor $(f_m) =
[Y_m]-a_{1,m}[Y_1]-\cdots -a_{n,m} [Y_n]$ for some integers
$a_{i,j}$.  By assumption, there is some $q$ (depending on $m$) such
that $\sigma^{qa}(f_m)$ has the same divisor as $f_m$ for all
integers $a$.  Thus we may pick some number $p$ such that for $m\in
[n+1, n+d+1]$ we have
$$
g_m :=\sigma^p(f_m)/f_m
$$
is in $\Gamma(Y, \mc{O}_Y)^*$.  It follows that the images of
$g_{n+1},g_{n+2},\ldots, g_{n+d+1}$ in $H$ have some nontrivial
relation.  That is, for some integers $b_j$ (not all zero), we have
$$
\alpha = \prod_{j=1}^{d+1} g_{n+j}^{b_j} \in \overline{F}^*.
$$
 Then
$$
f=\prod_{j=1}^{d+1} f_{n+j}^{b_j}
$$ is a rational function such that $\sigma^p(f) = \alpha f$.  Moreover, if $j$ is the largest index such
that $b_j\not =0$, then $f$ has a zero or pole in the dvr
corresponding to $Y_{n+j}$, so it is not in $\overline{F}$.

Notice that we used only finitely many of the $Y_i$ above.  Throwing
those away, we can repeat the process above and construct another
non-constant rational function $f'$ with $\sigma^q(f') = \alpha'
f'$, and such that the codimension-$1$ subsets appearing as zeroes
or poles in $(f')$ are completely different from those appearing in
$(f)$. Repeating this process, we can construct rational functions
$h_0, h_1, \dots, h_s$, where $s = \dim Y$ is the transcendence
degree of $F(Y)$ over $F$, satisfying $\sigma^p(h_i) = \alpha_i h_i$
for some fixed $p$, such that $(h_i) \neq 0$ and no codimension-$1$
irreducible subset of $Y$ appears in more than one divisor $(h_i)$.
The elements $h_0,\ldots, h_s$ are algebraically dependent over $F$,
so we may pick a nonzero polynomial $P(x_0,\ldots ,x_s) \in F[x_0,
\ldots, x_s]$ such that $P(h_0,\ldots ,h_s)=0$.  Choose such a $P$
with a minimal number of monomial terms with nonzero coefficient.
By construction, $\sigma^p(P(h_0,\ldots ,h_s))= P(\alpha_0
h_0,\ldots ,\alpha_s h_s)=0$.  If $x_0^{i_0}\cdots x_s^{i_s}$ and
$x_0^{j_0}\cdots x_s^{j_s}$ are distinct monomials with nonzero
coefficient in $P$, then we must have
$$
\prod_{\ell} \alpha_{\ell}^{i_{\ell}}=\prod_{\ell}
\alpha_{\ell}^{j_{\ell}},
$$
since otherwise we would contradict the minimality of $P$.  Thus
$$
h\ = \ \prod_{\ell} h_{\ell}^{i_{\ell}-j_{\ell}}
$$ is $\sigma^p$-invariant, and since the
$(h_{\ell})$ all involve different codimension-1 subsets, there is
no cancelation and so $(h) \neq 0$.   Thus $h \not \in
\overline{F}$, so the element $(h \otimes 1)$ of $F(Y) \otimes_F k
\subseteq k(X)$ is a rational function on $X$ which is
$\sigma^p$-invariant and not in $k$.

All statements in the theorem now follow immediately from
Theorem~\ref{unc-irr-thm}.
\end{proof}

\section{Growth type}
\label{growth-type-sec}

For the results of this section only, the cardinality of $k$ is not
relevant, and so we assume the following
\begin{standing-hypothesis}
\label{k-hyp3} Throughout this section, $k$ is an algebraically
closed field.
\end{standing-hypothesis}
\noindent Suppose that $A$ is one of the examples introduced in
Section~\ref{examples-sec}, with corresponding geometric data $(X,
\sigma)$ where $X$ is a quasi-projective scheme.
Theorem~\ref{main-ord-thm} suggests that there is some connection
between the growth of $A$ and the geometry of $(X, \sigma)$, and in
this section we begin to explain this connection.

In fact, we will work with birationally commutative algebras more
generally. Suppose that $A$ is a finitely generated connected
$\mb{N}$-graded $k$-algebra, which is a noetherian domain with
$Q_{\rm gr}(A) = K[t, t^{-1}; \sigma]$ where $K$ is a field. In
\cite[Theorem 1.4]{RZ}, it was shown that whether $A$ has finite
GK-dimension or not depends only on $(K, \sigma)$.   This idea was
explored further in \cite{Ro}, in the special case where $K/k$ is a
finitely generated extension with $\trdeg K/k = 2$; specifically, it
was shown how to exactly characterize which $(K, \sigma)$ correspond
to algebras $A$ of finite GK-dimension.  In this section, we review
the results of \cite{Ro} and prove some generalizations which are
useful to know in themselves, as well as being applicable to our
study of ordinary automorphisms in the next section.  In particular,
we show that the case where $\GK A < \infty$ above can in fact be
characterized by the condition $\GK K[t, t^{-1}; \sigma] < \infty$;
this extends \cite[Theorem 1.6]{RZ}. This will also allow us to
extend the theory of \cite{Ro} to $\mb{Z}$-graded algebras such as
skew and skew-Laurent extensions.

We now recall some needed definitions from \cite{Ro}.
\begin{definition}
\label{bir-data-def} Let $\sigma: K \to K$ be an automorphism of a
finitely generated field extension of $k$ with $\trdeg K/k = 2$. We
define associated \emph{growth data} $(\rho, j)$ with real number
$\rho \geq 1$ and integer $j \geq 0$, as follows. Pick any
nonsingular projective surface $X$ with $k(X) = K$, and consider the
birational map $\sigma: X \dra X$ induced by $\sigma: K \to K$. Let
$N^1(X)$ be the group of Cartier divisors on $X$ modulo numerical
equivalence.  Each power $\sigma^n$ induces a homomorphism
$(\sigma^n)^*: N^1(X) \to N^1(X)$, defined essentially by pullback
of divisors (see \cite[Definition 2.1]{Ro} and the discussion
following.)  The group $N^1(X)$ is free abelian of finite rank and
so each $(\sigma^n)^*$ corresponds to an integer matrix. Finally,
choosing any matrix norm, the sequence $f(n) = || (\sigma^n)^* ||$
is equivalent to the sequence $g(n) = n^j \rho^n$ for a unique
$(\rho, j)$ as above (independent of $X$), in the sense that
$\lim_{n \to \infty} f(n)/g(n) < \infty$ \cite[Lemma 2.12]{Ro}.
 Finally, we say that $(K, \sigma)$ (or $(X, \sigma)$)
has \emph{finite growth type} if the associated growth data has
$\rho =1$, and \emph{infinite growth type} if $\rho
> 1$.
\end{definition}

The main results of \cite{Ro} depend on a beautiful classification
of the possible values of $(\rho, j)$ that can occur in
Definition~\ref{bir-data-def}, which is due (in case $k = \mb{C}$)
to Diller and Favre \cite{DF}.  We recall in the next theorem the
part of the classification corresponding to finite growth type, for
which we need a few more definitions. Given a birational self-map of
a projective variety $\sigma: X \dra X$, we say that it is
\emph{conjugate} to a birational map $\tau: Y \dra Y$ if there is a
birational map $\theta: X \dra Y$ such that $\tau \theta = \theta
\sigma$ (as birational maps).  A surjective morphism $f: X \to C$
from a integral nonsingular surface $X$ to a nonsingular curve $C$
is called a \emph{fibration}, and it is called
\emph{$\sigma$-invariant} if there is an automorphism $\sigma': C
\to C$ such that $\sigma' f = f \sigma$. The fibration is
\emph{rational} if the generic fiber is a rational curve and
\emph{elliptic} if the generic fiber is elliptic.
\begin{theorem}
\label{DF-sum-thm} \cite[Theorem 0.2]{DF} Let $\sigma: X \dra X$ be
a birational map of an integral projective surface $X$ over
$\mb{C}$. Define the growth data $(\rho, j)$  of $(k(X), \sigma)$ as
in Definition~\ref{bir-data-def}.  If $\rho = 1$, then possibly
after replacing $\sigma: X \dra X$ with a conjugate map $\tau: Y
\dra Y$, where $Y$ is a nonsingular projective surface, we can
assume that exactly one of the following cases occurs:
\begin{enumerate}
\item $j = 0$, $\tau: Y \to Y$ is an automorphism, and
$(\tau^n)^*: N^1(Y) \to N^1(Y)$ is the identity map for some $n \geq
1$.
\item $j = 1$, and there is a $\tau$-invariant rational fibration $f: Y \to
C$.
\item $j = 2$, $\tau: Y \to Y$ is an automorphism, and
there is a $\tau$-invariant elliptic fibration $f: Y \to C$.
\end{enumerate}
\end{theorem}

Theorem~\ref{DF-sum-thm} will be used in a crucial way in the next
section, where we will use the detailed information the
classification provides to help us prove that if $(X, \sigma)$ has
finite growth type, then it is ordinary.   In the remainder of this
section, we prove an alternative characterization of the property of
finite growth type which includes the case of $\mb{Z}$-graded
algebras. Consider the $\mb{Z}$-graded ring $Q = K[t, t^{-1};
\sigma]$ where $K/k$ is a finitely generated field extension and $t$
has degree $1$.   We say that a $\mb{Z}$-graded $k$-subalgebra $A
\subseteq Q$ is \emph{big} in $Q$ if $A \neq A_0$ and there exist
finitely many homogeneous elements $a_i, b_i \in A$, where $a_i,
b_i$ have the same degree $d_i$ for each $i$, such that the elements
$\{ a_ib_i^{-1} \}$ generate $Q_0 = K$ as a field extension of $k$.
If $A$ is a graded Ore domain, it is clearly big in $Q$ if and only
if $Q_{\rm gr}(A)$ is equal to some Veronese ring of $Q$; the
definition of big allows us to avoid worrying about the Ore
condition.
\begin{proposition}
\label{growth-prop} Let $K/k$ be a finitely generated field
extension with $\trdeg K/k = 2$, where $k$ is algebraically closed,
and let $\sigma \in \aut_k K$. Put $Q = K[t, t^{-1}; \sigma]$. Then
the following are equivalent:
\begin{enumerate}
\item $(K, \sigma)$ has finite growth type;
\item $\GK Q  < \infty$;
\item There exists a $\mb{Z}$-graded big subalgebra $A \subseteq Q$ such that $\GK A < \infty$.
\end{enumerate}
\end{proposition}
\begin{proof}
$(1) \implies (2)$: If there exists a nonsingular projective surface
$X$ with $K = k(X)$ such that the corresponding induced map on $X$
is an \emph{automorphism} $\sigma: X \to X$, the result $\GK Q <
\infty$ has already been shown in \cite[Theorem 1.6]{RZ}.  We now
show that a similar argument works in general, using also the
techniques from \cite{Ro}.  (Thus this argument is needed only if
case (2) of Theorem~\ref{DF-sum-thm} occurs, but it does not seem to
help to restrict to this case.)

First, given any nonsingular projective surface $X$ with $K = k(X)$,
we have an induced birational map $\sigma: X \dra X$.  The
birational map $\sigma$ induces a homomorphism $\sigma^*: \Pic X \to
\Pic X$ which is essentially pullback of divisors \cite[Definition
2.1]{Ro}.  Now for an invertible sheaf $\mc{M} \subseteq \mc{K}$,
where $\mc{K}$ is the constant sheaf of rational functions on $X$,
we write $\mc{M}^{\sigma}$ for $\sigma^* \mc{M} \subseteq \mc{K}$.
We then set $\mc{M}_n = \mc{M} \otimes \mc{M}^{\sigma} \otimes \dots
\otimes \mc{M}^{\sigma^{n-1}} \subseteq \mc{K}$ for each $n \geq 1$.
For a $k$-subspace $V \subseteq K$, we also write $V^{\sigma} =
\sigma^{-1}(V)$.  Now by \cite[Theorem 2.10]{Ro}, we can choose $X$
initially so that $\sigma$ is \emph{stable} in the sense of
\cite[Definition 2.7]{Ro}; the important consequence of this is that
if $V = \HB^0(X, \mc{M}) \subseteq K$, then $VV^{\sigma} \dots
V^{\sigma^{n-1}} \subseteq \HB^0(X, \mc{M}_n)$ \cite[Lemma 4.1]{Ro}.
 If $\mc{M}$ is a very ample invertible
sheaf on $X$, possibly after replacing
$\mc{M}$ by a large tensor power we also have the following good
vanishing properties:  $\HB^i(X, \mc{M}_n^{\otimes m}) = 0$ for all
$i \geq 1$, $n \geq 1$, $m \geq 1$  \cite[Lemma 4.4]{Ro}.

We now follow the proof of \cite[Theorem 1.6]{RZ}.  The argument
there shows that the growth of any finitely generated subalgebra of
$Q$ is bounded above by the growth of the sequence $(2n + 1) \dim_k
\big(W W^{\sigma} \dots W^{\sigma^{2n}}\big)^n$ for some subspace $W
\subseteq K$ with $1 \in W$.  We may choose a very ample sheaf
$\mc{M} \subseteq \mc{K}$ with good vanishing properties as above,
and such that $W \subseteq \HB^0(X, \mc{M}) \subseteq K$ \cite[Lemma
5.2]{RZ}.  Then
\[
\big(W W^{\sigma} \dots W^{\sigma^{2n}}\big)^n \subseteq \HB^0(X,
\mc{M}_{2n+1}^{\otimes n}) \subseteq K.
\]
Since all of the sheaves $\mc{M}_{2n+1}^{\otimes n}$ have vanishing
higher cohomology, the Riemann Roch formula gives that $\dim_k
\HB^0(X, \mc{M}_{2n+1}^{\otimes n})$ grows like the intersection
numbers $(\mc{M}_{2n+1}^{\otimes n} . \mc{M}_{2n+1}^{\otimes n}
\otimes \omega^{-1})$, where $\omega$ is the canonical sheaf on $X$.
Finally, these intersection numbers have polynomial growth in $n$,
by \cite[Lemma 5.4, 5.6]{Ro}; in fact, they grow no faster than
$n^{j + 4}$, where $(1,j)$ is the growth data of $\sigma$ as in
Definition~\ref{bir-data-def}. So one has the even stronger result
that $\GK Q \leq j + 5$.

$(2) \implies (3)$: This is trivial.

$(3) \implies (1)$: We show the contrapositive, so suppose that $(K,
\sigma)$ has infinite growth type.  The idea is to reduce to the
$\mb{N}$-graded case, where the result has already been shown in
\cite{Ro}.   Suppose that $A \subseteq Q$ is a $\mb{Z}$-graded big
subalgebra, so $A \neq A_0$.   Choose homogeneous elements $a_i, b_i
\in A$, where $a_i, b_i \in A_{d_i}$ for each $i$, such that the
elements $\{ a_ib_i^{-1} \}$ generate $K$ as a field extension of
$k$.  Suppose that $A$ contains a nonzero homogeneous element of
positive degree, say $s \in A_d$ for some $d > 0$.  Then for $n \gg
0$, the elements $a_i' = a_is^n, b_i' = b_is^n$ all have positive
degree, and $A' = k \langle a'_1, \dots, a'_n, b'_1, \dots, b'_n
\rangle $ is now a connected finitely generated $\mb{N}$-graded big
subalgebra of $Q$ with $A' \subseteq A$.  By \cite[Theorem 1.1,
Theorem 7.1]{Ro}, since $(K, \sigma)$ has infinite growth type, $A'$
has exponential growth, so the same is true of $A$.  If $A$ contains
instead an element of negative degree, apply the same argument to
the graded ring $\wt{A} = \bigoplus_{n \in \mb{Z}} A_{-n}$.
\end{proof}

To conclude this section, we show that the property of finite growth
type is invariant under base extension.
\begin{lemma}
\label{fgt-ext-lem} Let $k \subseteq \ell$ with $\ell$ also
algebraically closed. Given an extension $k \subseteq K$ where
$\trdeg_k K = 2$, let $L$ be the field of fractions of $K \otimes_k
\ell$. Given an automorphism $\sigma \in \aut_k K$, consider the
induced automorphism $\sigma' \in \aut_{\ell} L$. Then $(K, \sigma)$
has finite growth type if and only if $(L, \sigma')$ does.
\end{lemma}
\begin{proof}
Note that $K \otimes_k \ell$ is a domain by Lemma~\ref{ac-ext-lem},
so we can indeed form the field of fractions $L$.  Let $Q = K[t,
t^{-1}; \sigma]$, $Q' = L[t, t^{-1}; \sigma']$, and $R = Q \otimes_k
\ell \cong (K \otimes_k \ell)[t, t^{-1}; \sigma']$.  Then $\GK_k Q =
\GK_{\ell} R$, and $R$ is a big $\ell$-subalgebra of $Q'$.  The
result now follows immediately from Proposition~\ref{growth-prop}.
\end{proof}

\section{Ordinary automorphisms}

\label{ord-auto-sec}

Suppose that $X$ is a scheme of finite type over $k$ with
automorphism $\sigma$.   The goal of this section is to study when
$\sigma$ is ordinary using the methods of algebraic geometry.  Thus
we assume Hypothesis~\ref{k-hyp2} throughout this section.

The most general method we present for proving that an automorphism
$\sigma$ is ordinary is the following:  if $\sigma$ is an element of
an algebraic group of automorphisms, then $\sigma$ is ordinary (see
Proposition~\ref{alg-gp-prop} below.)  This result applies to some
important special cases of automorphisms of varieties of arbitrary
dimension, for example any automorphism of a complete toric variety.
We then examine the case where $X$ is a surface in greater detail.
Here, we will need to apply much more detailed information about the
automorphism groups of surfaces to prove Theorem~\ref{main-ord-thm}
from the introduction, namely that automorphisms of finite growth
type are ordinary.

In this section, all divisors on $X$ will be Cartier divisors, and
we use $\sim$ for linear equivalence of divisors.  We start by
recording some of the more elementary observations about ordinary
automorphisms.  We will use the following easy lemma, whose proof we
leave to the reader, without comment below.  In particular, the
lemma shows that in studying ordinary automorphisms we can reduce to
the case of an integral scheme.

\begin{lemma}
\label{components-lem} Let $X$ be a finite type $k$-scheme with
$k$-automorphism $\sigma$.
\begin{enumerate}
\item $(X, \sigma)$ is ordinary if and only if $(X_{red}, \sigma)$
is.

\item For any $n \geq 1$, $(X, \sigma)$ is ordinary if and only if $(X,
\sigma^n)$ is ordinary.

\item Let $X$ be a reduced scheme with irreducible components $X_1, \dots, X_d$.  Then choosing any $n
\geq 1$ so that $\sigma^n(X_i) = X_i$ for all $i$, $(X, \sigma)$ is
ordinary if and only if $(X_i,  \sigma^n \vert_{X_i})$ is ordinary
for all $i$.
\end{enumerate}
\end{lemma}

Next, we see that the case of a curve is easy.  We remark that part
(1) of the next lemma also follows from
Theorem~\ref{inv-divisor-thm}, but we prefer to give a more direct
proof.

\begin{lemma}
\label{curves-lem} Let $\sigma: X \to X$ be an automorphism of a
finite-type $k$-scheme $X$.
\begin{enumerate}
\item If $\dim X \leq 1$, then $(X, \sigma)$ is ordinary.
\item If $X$ is $\sigma$-irreducible and $\dim X = 2$, 
then $(X, \sigma)$ is ordinary if and only if $(X, \sigma)$ has good
dense orbits.
\end{enumerate}
\end{lemma}
\begin{proof}
(1) By Lemma~\ref{components-lem}, we may assume $X$ is integral.
The case $\dim X = 0$ is trivial, so assume $\dim X = 1$. There is a
birational map $f: Y \dra X$ where $Y$ is the unique nonsingular
projective curve with $k(Y) = k(X)$, and an induced automorphism
$\tau: Y \to Y$ such that $f \tau = \sigma f$; it clearly suffices
to show that $(Y, \tau)$ is ordinary. As is well-known, any
automorphism of $Y$ is either of finite order or else its set of
periodic points consists of $0,1,$ or $2$ fixed points (for example,
see \cite[p.14]{AS1}). The result easily follows.

(2) Again we may reduce to the integral case, and the result is an
immediate consequence of part (1).
\end{proof}

The property of having good dense orbits passes between schemes
related by a dominant rational map.
\begin{lemma}
\label{dom-ord-lem} Let $X$ and $Y$ be integral finite-type
$k$-schemes with $\dim X = \dim Y$, and let $f: X \dra Y$ be a
dominant rational map. Suppose that $\sigma: X \to X$ and $\tau:Y
\to Y$ are automorphisms such that $f \sigma = \tau f$ (as rational
maps). Then $(X, \sigma)$ has good dense orbits if and only if $(Y,
\tau)$ has good dense orbits.
\end{lemma}
\begin{proof}
Let $U \subseteq X$ be the maximal open set on which $f$ is defined;
then $U$ is $\sigma$-invariant.  Now $(X, \sigma)$ clearly has good
dense orbits if and only if $(U, \sigma \vert_U)$ does, so we may
replace $X$ by its open subscheme $U$ and suppose from now on that
$f$ is a dominant morphism.

Let $V \subseteq X$ be the set of points lying on a dense
$\sigma$-orbit, and let $W \subseteq Y$ be the set of points which
lie on a dense $\tau$-orbit.  It follows easily from the facts that
$\dim X = \dim Y$ and $f$ is dominant that $f^{-1}(W) = V$.  Thus if
$(Y, \tau)$ has good dense orbits then so does $(X, \sigma)$.

Suppose conversely that $(X, \sigma)$ has good dense orbits.  Since
$f$ is dominant, $f(X)$ contains an open dense subset of $Y$;
letting $T = \overline{Y \setminus f(X)}$, it is clear that $T$ is a
proper closed $\tau$-invariant subset of $Y$.  Thus if $X$ has no
dense $\sigma$-orbit, then $Y$ has no dense $\tau$-orbit. Assume now
that $X$ does have a dense orbit, so $V$ is an open dense subset.
Putting $Z = X \setminus V$, we see that $\overline{f(Z)}$ is a
proper closed $\tau$-invariant subset in $Y$. Thus $W = Y \setminus
(T \cup \overline{f(Z)})$ is open and $(Y, \tau)$ has good dense
orbits in this case as well.
\end{proof}

Next, we turn to our most general methods for proving an
automorphism has good dense orbits.  We emphasize that for us an
algebraic group is  a reduced group scheme of finite type over the
ground field.
\begin{proposition}
\label{alg-gp-prop} Let $X$ be an integral scheme of finite type
over $k$ with automorphism $\sigma$.
\begin{enumerate}
\item If an algebraic group $G$ acts on $X$ with some $g \in G$ acting by $\sigma$, then $(X, \sigma)$ is
ordinary.

\item If $D$ is an ample Cartier divisor on $X$ with $(\sigma^n)^* D \sim D$ for some $n \geq 1$, then $(X, \sigma)$
is ordinary.

\item If $D$ is a big Cartier divisor on $X$ with $(\sigma^n)^* D \sim D$ for some $n \geq 1$, then $(X, \sigma)$ has good dense orbits.
\end{enumerate}
\end{proposition}
\begin{proof}
(1) By Corollary~\ref{gdo-ord-cor}, we need to show, given any
$\sigma$-irreducible closed subset $Z \subseteq X$, that $(Z, \sigma
\vert_Z)$ has good dense orbits.  Let $H = \bbar{ \{ g^n | n \in
\mb{Z}\}} \subseteq G$. Then $H$ acts on $X$, and the $H$-invariant
subschemes of $X$ are precisely the $\sigma$-invariant subschemes of
$X$. Since $H$ thus also acts on any $\sigma$-invariant $Z$, it is
enough to show that $(X, \sigma)$ has good dense orbits.  If $X$ has
no dense $\sigma$-orbits, this is true by definition.  If $x \in X$
has a dense $\sigma$-orbit, let $U = Hx$. Any orbit of an algebraic
group action is locally closed \cite[Lemma 2.3.3]{Sp}, so $U$ is
open inside $\bbar{U}=X$. As $H$ is abelian, $U$ is precisely the
set of those $y \in X$ with a dense $\sigma$-orbit.

(2) We may replace $\sigma$ by $\sigma^n$ without loss of
generality.  Then some multiple $mD$ of $D$ is very ample and
$\sigma^*(mD) \sim mD$. Taking the projective embedding $\phi: X \to
\mb{P}^d$ corresponding to the complete linear series $|mD|$, we see
that $\sigma$ lifts to an automorphism $\wt{\sigma}: \mb{P}^d \to
\mb{P}^d$.  Since the automorphism group $\operatorname{PGL}(d+1,k)$
of $\mb{P}^d$ is an algebraic group, the result follows from part
(1).

(3)  Again we replace $\sigma$ by $\sigma^n$.  In this case, for
some multiple $mD$, the rational map $\phi: X \dra \mb{P}^d$
associated to the complete linear series $|mD|$ has the property
$\dim Y = \dim X$, where $Y = \overline{\phi(X)}$. As in part (2),
there is a compatible automorphism $\wt{\sigma}: \mb{P}^d \to
\mb{P}^d$, so $(Y, \wt{\sigma} \vert_{Y})$ is ordinary by part (1).
Then $(X, \sigma)$ has good dense orbits by Lemma~\ref{dom-ord-lem}.
\end{proof}
\noindent We remark that an automorphism of a variety of general
type, which satisfies part (3) of the preceding lemma for the
canonical divisor, is ordinary for a more trivial reason: such a
variety has a finite automorphism group.

In general, it is not obvious how to check if an automorphism
$\sigma$ belongs to an algebraic group of automorphisms, so that
part (1) of the preceding proposition applies.  The following
general remarks help to clarify the issue.  Let $X$ be a projective
integral $k$-scheme.  Then $\Aut(X)$ has the natural structure of a
group scheme \cite[Proposition 2.3]{Han}, which is constructed as an
open subscheme of the Hilbert scheme $\operatorname{Hilb}(X \times
X)$. This Hilbert scheme is a countable disjoint union of projective
schemes, so $\Aut(X)$ is a countable disjoint union of finite type
schemes.  Then $\Aut^0(X)$, the irreducible component of $\Aut(X)$
containing the identity, is an algebraic group.  Clearly a sub-group
scheme of $\Aut(X)$ is an algebraic group if and only if it is
contained in finitely many components.

The techniques in the preceding proposition are enough to show that
a number of special kinds of automorphisms of varieties of arbitrary
dimension are ordinary.  We will not attempt to exhaust all of the
possibilities here, but we give in the next result some important
examples.  Note that we say that an automorphism of $X$ is
\emph{numerically trivial} if its induced action on $N^1(X)$ is the
identity map.
\begin{corollary} \label{alg-gp-cor}
Let $X$ be an integral scheme of finite type over $k$ with
$k$-automorphism $\sigma$.  Then $(X, \sigma)$ is ordinary in any of
the following cases:
\begin{enumerate}
\item $X$ is a complete simplicial toric variety.

\item $X$ is projective, $\Pic(X)$ is discrete and $\sigma^n$ is numerically
trivial for some $n \geq 1$.

\item $X$ is an abelian variety and $\sigma^n$ is numerically
trivial for some $n \geq 1$.

\item $X$ is a nonsingular
projective Fano variety.
\end{enumerate}
\end{corollary}
\begin{proof}
(1) By \cite[Corollary~4.7]{Cox}, the automorphism group of a
complete simplicial toric variety is an algebraic group.  Now apply
Theorem~\ref{alg-gp-prop}(1).

(2) Let $D$ be an any ample divisor on $X$.  We have that
$(\sigma^n)^*D \equiv D$, where $\equiv$ means numerical
equivalence; this implies that $(\sigma^n)^*(mD)$ is algebraically
equivalent to $mD$ for some $m \geq 1$ \cite[XIII,
Theorem~4.6]{SGA6}.
 Since $\Pic(X)$ is discrete, the Picard
variety $\Pic^0(X) \subseteq \Pic(X)$ of divisor classes
algebraically equivalent to $0$ is a point, so $(\sigma^n)^*(mD)
\sim mD$, and by Proposition~\ref{alg-gp-prop}(3) we obtain that
$(X, \sigma^n)$ is ordinary.  Thus $(X, \sigma)$ is also.

(3) We may assume that $\sigma$ is itself numerically trivial. We
verify that the same proof as in \cite[Proposition~1]{Do} shows that
if $\sigma$ is a numerically trivial automorphism on an abelian
variety $X$, then $\sigma$ is contained in an algebraic group of
automorphisms. For $x \in X$, let $T_x: X \to X$ be the translation
automorphism $y \mapsto y + x$. Let $L$ be a line bundle on $X$, and
let $\phi_L: X \to \Pic(X)$ be the morphism defined by $\phi_L(x) =
T^*_x(L) \otimes L^{-1}$.  As is shown in \cite[II.8]{M}, the image
of any $\phi_L$ is contained in $\Pic^0(X)$.

Now let $L$ be a very ample line bundle on $X$.  Since $\sigma$ is
numerically trivial, we have $\sigma^*L \otimes L^{-1} \in
\Pic^0(X)$, because numerical equivalence and algebraic equivalence
are the same on an Abelian variety \cite[Corollary, p.135]{La1}.
 In addition, \cite[II.8, Theorem~1]{M} shows that for an ample line bundle
$L$, the map $\phi_L: X \to \Pic^0(X)$ is surjective. Thus there is
some $x$ such that
\[ \sigma^*L \cong T^*_{x} L.\]
That is, $T_{-x}^* \sigma^* $ leaves the isomorphism class of $L$
invariant.   But the group of automorphisms of $X$ leaving $L$
invariant is an algebraic group $G \subseteq \Aut(X)$, by the proof
of Theorem~\ref{alg-gp-prop}(2); thus $\sigma$ lies in the right
coset $G T_{x}$.  The group of translation automorphisms of $X$ is
isomorphic to the irreducible algebraic group $X$, so it must be
contained in $\Aut^0(X)$.  Since $G$ is an algebraic group, it is
contained in some finite set of irreducible components of $\Aut(X)$;
the union of these components must be an algebraic group containing
$\sigma$.   Now apply Theorem~\ref{alg-gp-prop}(1).

(4) In this case, by definition the anti-canonical divisor $-K_X$ is
ample, so the result follows from Theorem~\ref{alg-gp-prop}(2).
\end{proof}

Now we turn from techniques that work for varieties of any dimension
and give our main result on the case of surfaces.  We note that the
proof of the next theorem depends in an essential way on a result
about the automorphism groups of surfaces, \cite[Proposition~1]{Do},
which is proved case-by-case by appealing to the classification of
projective surfaces.  We also use the classification of the possible
growth data of an automorphism with finite growth type from
\cite{DF}, as we reviewed in Theorem~\ref{DF-sum-thm}. Thus there is
a lot of highly non-trivial geometry specific to surfaces supporting
this result.
\begin{theorem}
\label{dim2-ord-thm} Let $k$ be an uncountable algebraically closed
field of characteristic $0$.   Let $\sigma: X \to X$ be an
automorphism of an integral quasi-projective $k$-scheme $X$ with
$\dim X = 2$. If $(X, \sigma)$ has finite growth type, then $(X,
\sigma)$ is ordinary.
\end{theorem}
\begin{proof}
By Lemma~\ref{curves-lem}(2), we need only prove that $(X, \sigma)$
has good dense orbits.  Moreover, we can reduce to the case $k =
\mb{C}$, by the following standard application of the Lefschetz
principle.
 There is some subfield $E \subseteq k$ which is a
finitely generated field extension of $\mb{Q}$ and such that $(X,
\sigma)$ is defined over $E$.  We can find an algebraically closed
uncountable field $F$ with $E \subseteq F \subseteq k$ such that $F$
is also isomorphic to a subfield of $\mb{C}$.  Then $(X, \sigma)$ is
also defined over $F$, in other words there is an $F$-scheme $Y$
with $\sigma':Y \to Y$ such that $X = Y \times_{\spec F} \spec k$
and $\sigma'$ induces $\sigma$.  We also have an extended scheme
$\wt{X} = Y \times_{\spec F} \spec \mb{C}$ with induced automorphism
$\wt{\sigma}$.  Since $F$ is algebraically closed, $\wt{X}$ is still
integral by Lemma~\ref{ac-ext-lem}.  Now by
Lemma~\ref{gdo-basechange-lem}, $(X, \sigma)$ has good dense orbits
if and only if $(\wt{X}, \wt{\sigma})$ does, and by
Lemma~\ref{fgt-ext-lem}, $(\wt{X}, \wt{\sigma})$ still has finite
growth type.  Thus, replacing $X$ by $\wt{X}$, from now on we can
assume that $k = \mb{C}$.

By assumption, in the growth data $(\rho, j)$ of $\sigma: k(X) \to
k(X)$ we have $\rho = 1$.  We use the classification in
Theorem~\ref{DF-sum-thm}, which shows in particular that we only
need to consider the cases $j = 0,1,2$.  Using
Lemma~\ref{dom-ord-lem}, if convenient we can replace $\sigma: X \to
X$ with a conjugate automorphism $\tau: Y \to Y$ and it is enough to
prove that $(Y, \tau)$ has good dense orbits.

In case $j = 0$, using Theorem~\ref{DF-sum-thm} we change to a
conjugate automorphism $\tau: Y \to Y$ where $Y$ is nonsingular
projective and some power of $\tau$ is numerically trivial. Without
loss of generality, we may assume that $\tau$ is numerically
trivial.

 Let $\operatorname{O}(\operatorname{N}^1(Y))$
be the group of linear transformations on the lattice
$\operatorname{N}^1(Y)$ that preserve the intersection form.
Considering the action of automorphisms on $\operatorname{N}^1(Y)$
by pushforward, we have group a homomorphism
\[ r:  \Aut(Y) \to \operatorname{O}(\operatorname{N}^1(Y))\]
which induces (since $\Aut^0(Y)$ is irreducible and
$\operatorname{O}(\operatorname{N}^1(Y))$ is discrete) a
homomorphism
\[ \bbar{r}: \Aut(Y)/\Aut^0(Y) \to \operatorname{O}(\operatorname{N}^1(Y)). \]
In the case at hand, $\tau \in \Ker(r)$. By
\cite[Proposition~1]{Do}, $\Ker(\bbar{r})$ is finite; therefore,
$\tau$ is an element of a finite extension of the algebraic group
$\Aut^0(Y)$, namely $\Ker(r)$, which is an algebraic group. Then
$(Y, \tau)$ has good dense orbits by
Proposition~\ref{alg-gp-prop}(1).

The arguments for the remaining cases $j = 1$, $j = 2$ are similar
to each other.  Using Theorem~\ref{DF-sum-thm}, if $j = 1$, then
there is a birational map $g: X \dra Z$ for some nonsingular
projective surface $Z$ such that passing to the conjugate $\tau: Z
\dra Z$ of $\sigma$, there is a $\tau$-invariant rational fibration
$f: Z \to C$ for some  curve $C$. Consider the rational map $h = fg:
X \dra C$. The locus where $h$ is defined is a $\sigma$-invariant
open subset $Y \subseteq X$, and so we can consider the morphism
$h': Y \to C$; then it is enough to prove that $(Y, \sigma \vert_Y)$
has good dense orbits. In case $j = 2$, Theorem~\ref{DF-sum-thm}
immediately gives us a map $\tau: Y \to Y$ conjugate to $\sigma: X
\to X$, where $Y$ is nonsingular projective, and such that there is
a $\tau$-invariant elliptic fibration $f: Y \to C$ for some curve
$C$.

Thus to finish both cases, it is enough to consider an integral
quasi-projective surface $Y$ with automorphism $\tau: Y \to Y$,
together with a $\tau$-invariant fibration $f: Y \to C$ for some
curve $C$, and prove that $(Y, \tau)$ has good dense orbits. By
assumption, there is a compatible automorphism $\theta: C \to C$
with $f \tau = \theta f$.  If $(Y, \tau)$ has no dense orbit, we are
done, so assume this is not the case.  Then $\theta$ must have
infinite order, so $\theta$ has finitely many periodic points on $C$
by Lemma~\ref{curves-lem}.  The union of the fibers in $Y$ over
those points is a proper $\tau$-invariant closed subset of $Y$ which
must contain all of the $\tau$-periodic points of $Y$. Again since
we assume $(Y, \tau)$ has a dense orbit, by
Theorem~\ref{inv-divisor-thm} there are finitely many
$\tau$-periodic curves in $Y$.  Letting $W$ be the union of these
curves, we also know that $W$ contains all of the $\tau$-periodic
points of $Y$, so we conclude that $U = Y \setminus W$ is the set of
points lying on dense $\tau$-orbits, and $U$ is open as required.
\end{proof}
\begin{remark}
The restriction $\cha k = 0$ in Theorem~\ref{dim2-ord-thm} is
probably unnecessary, but removing it would likely take significant
work.  To do so, one would need to verify that
\cite[Proposition~1]{Do} and Theorem~\ref{DF-sum-thm} both work over
algebraically closed fields of positive characteristic.  (In
\cite[Theorem 3.2]{Ro}, most of Theorem~\ref{DF-sum-thm} is shown to
hold over any algebraically closed field, but not the conclusion
that there is a conjugate surface with a $\tau$-invariant elliptic
fibration in case $j = 2$.)
\end{remark}

We conjecture that the converse to the preceding theorem also holds:
namely, that on a quasi-projective integral surface $X$, an
automorphism $\sigma$ with infinite growth type never has good dense
orbits.  It is not too hard to show that such an automorphism cannot
have a $\sigma$-invariant rational function, and thus must have
finitely many $\sigma$-periodic curves by
Theorem~\ref{inv-divisor-thm}. Thus the failure of good dense orbits
is equivalent in this case with having a countably infinite number
of periodic points which do not lie on periodic curves.  The result
\cite[Theorem 0.6]{DF} does show that if $X$ is projective, $\sigma$
has infinite growth type \emph{and no curves of periodic points},
then $\sigma$ has a countably infinite number of periodic points. We
see no obvious argument that works in general, though, and so we do
not pursue this question further here. For the special case of $X =
\mb{A}^2$, however, see Theorem~\ref{A2-case-thm} below.

\section{Summary theorems and examples}
\label{summary-sec}

We assume Hypothesis~\ref{k-hyp2} throughout this section.   We are
now ready to prove Theorem~\ref{main-thm} from the Introduction,
which shows that the Dixmier-Moeglin equivalence often does hold for
the examples in Section~\ref{examples-sec}. At this point, the proof
is simply a matter of piecing together results we have already
proved.   We then give a number of illustrative examples of skew and
skew-Laurent rings which do not satisfy the DM-equivalence. These
examples are well-known, but it is interesting to see how they fit
into our geometric framework.

\begin{theorem}
\label{main-DM-thm} Let $k$ be an algebraically closed uncountable
field, and let $A$ be one of the following $k$-algebras:
\begin{enumerate}
\item A twisted homogeneous coordinate ring $A = B = B(X, \mc{L}, \sigma)$
as in Example~\ref{thcr-ex}, so $X$ is projective with automorphism
$\sigma: X \to X$ and $\mc{L}$ is $\sigma$-ample; or
\item either a skew polynomial ring $A =U = S[t; \sigma]$, or skew-Laurent ring $A =T = S[t, t^{-1}; \sigma]$ as in
Example~\ref{skew-ex}, so $S$ is a finitely generated commutative
$k$-algebra.  In this case, we set $X = \spec S$ and write $\sigma:
X \to X$ for the induced scheme automorphism.
\end{enumerate}
In all cases, if $(X, \sigma)$ is ordinary then $A$ satisfies the
DM-equivalence.  In particular, this always holds when $\cha k = 0$,
$\dim X \leq 2$, and $\GK A < \infty$ (the last condition being
automatic in case (1)).
\end{theorem}
\begin{proof}
If $(X, \sigma)$ is ordinary, then $(X, \sigma)$ is \ord\ by
Corollary~\ref{gdo-ord-cor}.  Then $A$ satisfies the DM-equivalence
by Proposition~\ref{geom-primes-prop}.

Suppose now that $\GK A < \infty$.  This is always true in case (1),
since $\mc{L}$ is $\sigma$-ample, by \cite[Theorem 1.2]{Ke1}. Choose
an $n \geq 1$ such that $\sigma^n$ restricts to an automorphism of
each irreducible component $X_{\alpha}$ of $X$, and use the same
name $\sigma^n$ for the restriction $\sigma^n \vert_{X_{\alpha}}$.
In case (1), $\mc{L}_n$ is $\sigma^n$ ample on $X$, and so $\mc{L}_n
\vert_{X_{\alpha}}$ is $\sigma^n$-ample on $X_{\alpha}$.  Thus the
ring $A' = B(X_{\alpha}, \mc{L}_n \vert_{X_{\alpha}}, \sigma^n)$
also has $\GK A' < \infty$. In case (2), putting $X_{\alpha} = \spec
S_{\alpha}$, then $A' = S_{\alpha}[t, t^{-1}; \sigma^n]$ is a factor
algebra of the $n$th Veronese ring $A^{(n)} = S[t, t^{-1};
\sigma^n]$ of $A$, so $\GK A' < \infty$ in this case also.   Since
in either case $A'$ is a noetherian domain of finite GK-dimension,
with graded quotient ring $k(X_{\alpha})[t, t^{-1}; \sigma^n]$, by
Proposition~\ref{growth-prop} $(k(X_{\alpha}), \sigma^n)$ has finite
growth type for each $\alpha$. Now if we know in addition that $\cha
k = 0$ and $\dim X \leq 2$, then $(X_{\alpha}, \sigma^n)$ is
ordinary for each $\alpha$ by Lemma~\ref{curves-lem} (in case $\dim
X_{\alpha} \leq 1$) or Theorem~\ref{dim2-ord-thm} (in case $\dim
X_{\alpha} = 2$). Then $(X, \sigma)$ is ordinary by
Lemma~\ref{components-lem}.
\end{proof}

\begin{remark}
Suppose that $A$ is a prime twisted homogeneous coordinate ring as
in case (1) of Theorem~\ref{main-DM-thm}, for which $(X, \sigma)$ is
indeed ordinary.  Suppose that $A$ is primitive.  Then it is also
clear how to construct a faithful simple module for $A$.  Since
$\spec A$ is \ord\ by Proposition~\ref{geom-primes-prop}, it must be
that $A$ has finitely many height one primes, say $P_1, P_2, \dots
P_n$, and these are necessarily homogeneous by the proof of
Proposition~\ref{hord-DM-prop}. Choose any $0 \neq x \in \bigcap
P_i$ which is homogeneous of positive degree, and let $M$ be any
maximal right ideal of $A$ containing $x-1$ (note that $x-1$ is not
a unit, as $A$ is $\mb{N}$-graded.) Then $A/M$ is a faithful simple
right $A$-module.
\end{remark}

We note next that our techniques allow us to give some partial
results on the question of the DM-equivalence for twisted homogenous
coordinate rings of higher-dimensional varieties.
\begin{proposition}\label{cor-threefold}
Assume that $\cha k = 0$, and let $X$ be an integral projective
three-fold with automorphism $\sigma$ and $\sigma$-ample line bundle
$\Lsh$. If $(X, \sigma)$ has good dense orbits, then $(X, \sigma)$
is ordinary and $B = B(X, \Lsh, \sigma)$ satisfies the
Dixmier-Moeglin equivalence.
\end{proposition}
\begin{proof}
It is enough to show that for any proper $\sigma$-irreducible subset
$Z$ of $X$, $(Z, \sigma \vert_Z)$ has good dense orbits, and then
apply Corollary~\ref{gdo-ord-cor} and
Proposition~\ref{geom-primes-prop}.  Since $\mc{L} \vert_Z$ is
$\sigma \vert_Z$-ample and $\dim Z \leq 2$, the proof of
Theorem~\ref{main-DM-thm} shows that $(Z, \sigma \vert_Z)$ has good
dense orbits as required.
\end{proof}

If $\sigma: X \to X$ is an automorphism of a projective scheme $X$
of arbitrary dimension, then $\sigma$ is called
\emph{quasi-unipotent} if the induced action $\sigma^*: N^1(X) \to
N^1(X)$ has eigenvalues which are roots of unity; this is the
correct analog of $\sigma$ having finite growth type, as defined in
Section~\ref{growth-type-sec} for a surface, if $X$ is projective of
higher dimension.  Keeler proved that an automorphism $\sigma$ is
quasi-unipotent if and only if $X$ has a $\sigma$-ample invertible
sheaf $\mc{L}$ \cite[Theorem 1.2]{Ke1}.  If this holds, then given
any $\sigma$-invariant subscheme $Z \subseteq X$, since $\mc{L}
\vert_Z$ is $\sigma \vert_Z$-ample, it follows that $\sigma \vert_Z$
is also quasi-unipotent.

\begin{proposition}\label{thcr-DM-prop}
 $(X, \sigma)$ is
ordinary for all projective $k$-schemes $X$ and quasi-unipotent
automorphisms $\sigma$, unless there exists such an $(X, \sigma)$
with the following additional properties:  $X$ is integral, $\sigma$
has a dense orbit, and $X$ has precisely a countably infinite number
of maximal $\sigma$-irreducible closed subsets, all but finitely
many of which have codimension $\geq 2$ in $X$.
\end{proposition}
\begin{proof}
Suppose that $(X, \sigma)$ is not ordinary, where $\sigma: X \to X$
is quasi-unipotent.  Since quasi-unipotence restricts to
$\sigma$-invariant subschemes as we observed above, by noetherian
induction we may assume that $X$ is $\sigma$-irreducible and that
$(X, \sigma)$ does not have good dense orbits.  Replacing $\sigma$
by a power and passing to an irreducible component, we may assume
that $X$ is integral.  Then the remaining claimed properties of
$\sigma$ are implied immediately by Theorem~\ref{unc-irr-thm} and
Theorem~\ref{inv-divisor-thm}.
\end{proof}

Since if $(X, \sigma)$ is ordinary, then $B(X, \mc{L}, \sigma)$
satisfies the DM-equivalence for any $\sigma$-ample $\mc{L}$
(Theorem~\ref{main-DM-thm}), the previous proposition shows that any
counterexamples to the DM-equivalence for twisted homogeneous
coordinate rings must come from automorphisms of a very restrictive
and probably nonexistent kind.  Thus we venture the following.
\begin{conjecture}
Given a projective $k$-scheme $X$ of any dimension and automorphism
$\sigma: X \to X$, with $\sigma$-ample sheaf $\mc{L}$, we conjecture
that $B(X, \mc{L}, \sigma)$ satisfies the DM-equivalence.
\end{conjecture}

In the remainder of the paper, we discuss the case of skew and
skew-Laurent extensions further.  Of course, in case (2) of
Theorem~\ref{main-DM-thm}, $\GK A < \infty$ does not always hold,
and in this case one recovers some the standard examples which do
not satisfy the DM-equivalence. In fact, we suspect that the skew
and skew-Laurent extensions $T$ and $U$ never satisfy the
DM-equivalence in the case they have infinite GK-dimension; see
Theorem~\ref{A2-case-thm} below for some evidence for this
contention. We now review some interesting examples.

\begin{example}
\label{T-prim-char} Let $T = S[t, t^{-1}; \sigma]$ be as in
Theorem~\ref{main-DM-thm}(2), where we assume for simplicity that
$S$ is a domain.  Then the primitivity of the ring $T$ is
characterized by Jordan in \cite[Theorem 5.10, Proposition 2.9]{Jo}.
Namely, $T$ is primitive if and only if $(X, \sigma)$ has a dense
orbit.  In fact, Jordan's results allow non-finitely-generated
commutative algebras $S$; in the special case at hand, we can easily
see why this characterization of primitivity holds using our
previous results, as follows. If $p \in X$ lies on a dense orbit,
then letting $\mf{m}$ be the maximal ideal of $S$ corresponding to
the point $p$, it is easy to check that $\bigoplus_{n \in \mb{Z}}
(S/\mf{m})t^n$ is a faithful simple right $T$-module.  On the other
hand, if $(X, \sigma)$ has no dense orbit, then
Proposition~\ref{unc-irr-thm} implies that $X$ has a non-constant
$\sigma$-invariant rational function $f$, so $f$ is in the center of
$Q_{\rm gr}(T) = k(X)[t, t^{-1}; \sigma]$ and $(0)$ is not a
rational ideal.  Then $T$ is not primitive, as we saw in the proof
of Lemma~\ref{height-one-lem}.
\end{example}

\begin{example}
Consider the situation of Theorem~\ref{main-DM-thm}(2) with $S =
k[u,v, u^{-1}, v^{-1}]$, so $X \cong (\mb{A}^1 - \{0 \})^2$, and where
$\sigma: S \to S$ is defined by $\sigma(u) = u^2v, \sigma(v) = uv$.
This was the original counterexample given by Lorenz in
\cite[4.3]{Lo} which showed that $T = S[t, t^{-1}; \sigma]$ fails the Dixmier-Moeglin
equivalence, since the prime ideal $(0)$ is primitive but not locally
closed. It is easy to check directly that $T$ has exponential
growth.  Theorem~\ref{main-DM-thm} shows that the lack of finite
GK-dimension for $T$ is a crucial property of this example.
\end{example}

In the case $X = \mb{A}^2$, we can characterize completely which
automorphisms are ordinary, extending the analysis of
\cite[Proposition 7.8]{Jo}, which showed which automorphisms have a
dense orbit.  This gives a large number of examples that do not
satisfy the DM-equivalence.
\begin{theorem}
\label{A2-case-thm} Let $k$ be uncountable and algebraically closed
with  $\cha k = 0$, and consider the situation of
Theorem~\ref{main-DM-thm}(2) with $S =k[u,v]$ and $X = \spec S =
\mb{A}^2$, so $T = S[t, t^{-1}; \sigma]$. Then the following are
equivalent:
\begin{enumerate}
\item $\GK T < \infty$;
\item  $(\mb{A}^2, \sigma)$ is ordinary;
\item $T$ satisfies the DM-equivalence.
\end{enumerate}
\end{theorem}
\begin{proof}
$(1) \implies (2) \implies (3)$ is part of
Theorem~\ref{main-DM-thm}, so we need only prove $(3) \implies (1)$,
for which we prove the contrapositive. Suppose that $\GK T =
\infty$; then $(k(u,v), \sigma)$ has infinite growth type, by
Proposition~\ref{growth-prop}. We will show that $T$ is primitive
but that the ideal $(0)$ is not locally closed. (By the first
paragraph of Lemma~\ref{height-one-lem}, $(0)$ will, however, be
rational.)

We claim that $(\mb{A}^2, \sigma)$ is not ordinary, or equivalently
by Lemma~\ref{curves-lem} that $(\mb{A}^2, \sigma)$ does not have
good dense orbits. By the same Lefschetz principle argument as in
the proof of Theorem~\ref{dim2-ord-thm}, using
Lemma~\ref{gdo-basechange-lem} we may reduce to the case $k =
\mb{C}$.  But the dynamics of polynomial automorphisms of $\mb{C}^2$
are well-understood.  Every such automorphism is conjugate in the
group of polynomial automorphisms either to an \emph{elementary
automorphism} of the form $\tau(z,w) = (\alpha z + p(w), \beta w +
\gamma)$ for some polynomial $p(w)$ and constants $\alpha, \beta$;
or else to a \emph{Henon map}, which is a composition of Henon
automorphisms of the form $\tau(z,w) = (p(z) - aw, z)$ with $\deg
p(z) \geq 2$ \cite[Theorem 2.6]{FM}.  Passing to a conjugate clearly
does not affect the property of having good dense orbits. Now if
$\tau$ is an elementary automorphism, then the degrees of the
polynomials in the coordinates in the formula for $\tau^n$ stay
bounded for all $n$, and it easily follows that such an automorphism
has growth data $\rho = 1, j = 1$ in Definition~\ref{bir-data-def}.
In particular, such automorphisms have finite growth type.  Thus we
may assume without loss of generality that $\sigma$ is a Henon map.
Now the first paragraph of the proof of \cite[Proposition 7.8]{Jo}
shows that $(\mb{A}^2, \sigma)$ has a countably infinite set of
periodic points, but no periodic curves, so clearly $(\mb{A}^2,
\sigma)$ does not have good dense orbits, and the claim is proved.

Working again over our original field $k$, by
Theorem~\ref{unc-irr-thm} $\sigma$ must have a dense orbit, and the
number of its maximal $\sigma$-irreducible subsets is countably
infinite.  Thus $T$ is primitive by Example~\ref{T-prim-char}, and
$T$ must have a countably infinite number of h-height one primes by
Lemma~\ref{skew-primes-lem}; in fact, these are height one primes by
the proof of Lemma~\ref{bc-dichotomy-lem}, and thus  $(0)$ is not
locally closed.
\end{proof}

All of the examples above concerned the skew-Laurent extension $T = S[t, t^{-1}; \sigma]$.  The
skew extension case $U = S[t; \sigma]$ behaves differently in certain ways.  We close by discussing
an example of Jordan in the context of our results.
\begin{example}
\label{U-ex} Consider Theorem~\ref{main-DM-thm}(2), with $S$ a
domain for simplicity. The primitivity of the ring $U = S[t;
\sigma]$ is characterized in \cite{LM}, as follows.  The
automorphism $\sigma: X \to X$ is called \emph{special} if there
exists a proper closed subset $Y \subsetneq X$ such that every
proper irreducible $\sigma$-periodic subset $Z$ of $X$ satisfies $Z
\subseteq \sigma^n(Y)$ for some $n \in \mb{Z}$. By \cite[Theorem
3.10]{LM}, $U$ is primitive if and only if $\sigma$ is special.
\end{example}

Suppose in Example~\ref{U-ex} that $(X, \sigma)$ has good dense
orbits.  If $\sigma$ has no dense orbit, then the automorphism
$\sigma: X \to X$ is non-special (using
Lemma~\ref{amitsur-trick-lem}). If $\sigma$ has a dense orbit, then
clearly $\sigma$ is special (take $Y$ to be the union of the
finitely many maximal $\sigma$-irreducibles.) Thus in this case, $U
= S[t; \sigma]$ is primitive if and only if $\sigma$ has a dense
orbit, just as in the case of the skew-Laurent ring $T$. However, in
the cases where $(X, \sigma)$ does not have good dense orbits (for
example, the Henon maps of $X = \mb{A}^2$ as in
Theorem~\ref{A2-case-thm}), it seems to be very difficult to verify
whether or not $\sigma$ is special. Due to the work of Jordan, we do
have one interesting example.
\begin{example}
Let $k = \mb{C}$ and $S = \mb{C}[u,v, u^{-1}, v^{-1}]$, with
$\sigma: S \to S$ defined by $\sigma(u) = v, \sigma(v) = uv^{-1}$.
Then \cite[Propositions 7.11-13]{Jo} shows that $X$ has no
$\sigma$-periodic curves, a countably infinite set of
$\sigma$-periodic points, and that $\sigma$ is \emph{not} special.
Thus, as is also pointed out in \cite[Example 3.11]{LM}, for this
automorphism we have that $U = S[t ; \sigma]$ is not primitive,
whereas $T = S[t, t^{-1}; \sigma]$ is primitive.  Clearly $(\spec S,
\sigma)$ does not have good dense orbits, and so necessarily
$\sigma$ has infinite growth type (Theorem~\ref{dim2-ord-thm}) and
$T$ and $U$ have exponential growth  by
Proposition~\ref{growth-prop}.

We note that the Dixmier-Moeglin equivalence fails for both rings
$T$ and $U$, but in different ways.  Since the two rings have the
same Goldie quotient ring, and we saw that $(0)$ is a rational ideal
of $T$ in Theorem~\ref{A2-case-thm}, $(0)$ is a rational ideal of
$U$. Also, a similar argument as in the last sentence of
Theorem~\ref{A2-case-thm} shows that since $\sigma$ has a countably
infinite number of maximal $\sigma$-irreducible subsets, then $(0)$
is not locally closed in $\spec U$ either.  To summarize, in $T$,
$(0)$ is primitive, rational, but not locally closed; whereas in
$U$, $(0)$ is not primitive and not locally closed, but it is
rational.
\end{example}

We wonder if every automorphism with infinite growth type on an
affine surface over $k$ must fail to be special.  Even in the case
of $\mb{A}^2$, nothing seems to be known about this question beyond
Jordan's example. The theory of the dynamics of maps of the affine
plane has developed much in recent years; perhaps there are
applications of that rich theory to the question of speciality.

We close with an intriguing question.
\begin{question}
\label{GK-DM-ques}
Theorem~\ref{main-DM-thm} is quite suggestive
that there may be a deeper relationship between the growth of an
algebra and the Dixmier-Moeglin equivalence. Is there some more
general ring-theoretic argument that shows that algebras of finite
GK-dimension (perhaps with certain extra hypotheses) have an \ord\
prime spectrum?
\end{question}
This seems to be a deep question; a positive answer would imply, for
example, that a finitely generated noetherian k-algebra $A$ over an
uncountable field $k$, which is a domain of GK-dimension $2$, is
either PI or primitive.  (This is the well-known Dichotomy
Conjecture of Small.)  For, if $A$ has finitely many height one
primes, it is primitive by Lemma~\ref{height-one-lem}.  On the other
hand, if $A$ has uncountably height one primes, then there are
infinitely many such, say  $P_1, P_2, \dots$, such that the
PI-degrees of the GK-1 factor rings $A/P_i$ are all equal to a fixed
integer $d$. Then as $\bigcap P_i = 0$, $A$ embeds in $\prod A/P_i$
and thus is PI of degree $d$ as well.

\providecommand{\bysame}{\leavevmode\hbox
to3em{\hrulefill}\thinspace}
\providecommand{\MR}{\relax\ifhmode\unskip\space\fi MR }
\providecommand{\MRhref}[2]{%
  \href{http://www.ams.org/mathscinet-getitem?mr=#1}{#2}
} \providecommand{\href}[2]{#2}


\begin{thebibliography}{SGA6}

\bibitem[AS1]{AS1}
M.~Artin and J.~T. Stafford, \emph{Noncommutative graded domains
with quadratic
  growth}, Invent. Math. \textbf{122} (1995), no.~2, 231--276.

\bibitem[ATV1]{ATV1}
M.~Artin, J.~Tate, and M.~Van~den Bergh, \emph{Some algebras
associated to
  automorphisms of elliptic curves}, The Grothendieck Festschrift, Vol.\ I,
  Birkh\"auser Boston, Boston, MA, 1990, pp.~33--85.

\bibitem[AV]{AV}
M.~Artin and M.~Van~den Bergh, \emph{Twisted homogeneous coordinate
rings}, J.
  Algebra \textbf{133} (1990), no.~2, 249--271.

\bibitem[Bell]{Bell}
Jason~P. Bell, \emph{Noetherian algebras over algebraically closed
fields}, J.
  Algebra \textbf{310} (2007), no.~1, 148--155.

\bibitem[Ber]{Ber}
George~M. Bergman, \emph{Zero-divisors in tensor products},
Noncommutative ring
  theory (Internat. Conf., Kent State Univ., Kent, Ohio, 1975), Springer,
  Berlin, 1976, pp.~32--82. Lecture Notes in Math., Vol. 545.

\bibitem[SGA6]{SGA6}
Th\'eorie des intersections et th\'eor\`eme de Riemann-Roch.
(French) S\'eminaire de G\'eom\'etrie Alg\'ebrique du Bois-Marie
1966--1967 (SGA 6). Dirig\'e par P. Berthelot, A. Grothendieck et L.
Illusie. Avec la collaboration de D. Ferrand, J. P. Jouanolou, O.
Jussila, S. Kleiman, M. Raynaud et J. P. Serre. Lecture Notes in
Mathematics, Vol. 225. Springer-Verlag, Berlin-New York, 1971.

\bibitem[BG]{BG}
Ken~A. Brown and Ken~R. Goodearl, \emph{Lectures on algebraic
quantum groups},
  Advanced Courses in Mathematics. CRM Barcelona, Birkh\"auser Verlag, Basel,
  2002.


\bibitem[Cox]{Cox}
David~A. Cox, \emph{The homogeneous coordinate ring of a toric
variety}, J.
  Algebraic Geom. \textbf{4} (1995), no.~1, 17--50.

\bibitem[DF]{DF}
J.~Diller and C.~Favre, \emph{Dynamics of bimeromorphic maps of
surfaces},
  Amer. J. Math. \textbf{123} (2001), no.~6, 1135--1169.


\bibitem[Do]{Do}
I.~Dolgachev, \emph{Infinite {C}oxeter groups and automorphisms of
algebraic
  surfaces}, The {L}efschetz centennial conference, {P}art {I} ({M}exico
  {C}ity, 1984), Contemp. Math., vol.~58, Amer. Math. Soc., Providence, RI,
  1986, pp.~91--106.

\bibitem[FM]{FM}
Shmuel Friedland and John Milnor, \emph{Dynamical properties of
plane
  polynomial automorphisms}, Ergodic Theory Dynam. Systems \textbf{9} (1989),
  no.~1, 67--99.

\bibitem[GS]{GS}
K.~R. Goodearl and J.~T. Stafford, \emph{The graded version of
{G}oldie's
  theorem}, Algebra and its applications (Athens, OH, 1999), Contemp. Math.,
  vol. 259, Amer. Math. Soc., Providence, RI, 2000, pp.~237--240.

\bibitem[Ha]{Ha}
Robin Hartshorne, \emph{Algebraic geometry}, Springer-Verlag, New
York, 1977,
  Graduate Texts in Mathematics, No. 52.

\bibitem[Han]{Han}
Masaki Hanamura, \emph{On the birational automorphism groups of
algebraic
  varieties}, Compositio Math. \textbf{63} (1987), no.~1, 123--142.


\bibitem[Ir]{Ir}
Ronald~S. Irving, \emph{Primitive ideals of certain {N}oetherian
algebras},
  Math. Z. \textbf{169} (1979), no.~1, 77--92.


\bibitem[Jo]{Jo}
David~A. Jordan, \emph{Primitivity in skew {L}aurent polynomial
rings and
  related rings}, Math. Z. \textbf{213} (1993), no.~3, 353--371.

\bibitem[Ke1]{Ke1}
Dennis~S. Keeler, \emph{Criteria for $\sigma$-ampleness}, J. Amer.
Math. Soc.
  \textbf{13} (2000), no.~3, 517--532.


\bibitem[La1]{La1}
Serge Lang, \emph{Abelian varieties}, Springer-Verlag, New York,
1983, Reprint of the 1959 original.

\bibitem[La2]{La2}
Serge Lang, \emph{Fundamentals of {D}iophantine geometry},
Springer-Verlag, New York, 1983.

\bibitem[Lz1]{Lz1}
R.\ Lazarsfeld, \emph{Positivity in algebraic geometry. {I}},
Ergebnisse der Mathematik und ihrer Grenzgebiete. 3. Folge. A Series
of Modern Surveys in Mathematics, vol.~48, Springer-Verlag, Berlin,
2004, Classical setting: line bundles and linear series.

\bibitem[LM]{LM}
Andr{\'e} Leroy and Jerzy Matczuk, \emph{Primitivity of skew
polynomial and
  skew {L}aurent polynomial rings}, Comm. Algebra \textbf{24} (1996), no.~7,
  2271--2284.

\bibitem[Lo]{Lo}
Martin Lorenz, \emph{Primitive ideals of group algebras of
supersoluble
  groups}, Math. Ann. \textbf{225} (1977), no.~2, 115--122.


\bibitem[Ma]{Ma}
Hideyuki Matsumura, \emph{Commutative ring theory}, Cambridge
Studies in
  Advanced Mathematics, vol.~8, Cambridge University Press, Cambridge, 1986,
  Translated from the Japanese by M. Reid.

\bibitem[MR]{MR}
J.~C. McConnell and J.~C. Robson, \emph{Noncommutative {N}oetherian
rings},
  revised ed., American Mathematical Society, Providence, RI, 2001.


\bibitem[M]{M}
D.\ Mumford, \emph{Abelian Varieties}, Tata Institute of Fundamental
Research
  Studies in Mathematics, No.\ 5,  Tata Institute of
  Fundamental Research, Bombay, 1970.



\bibitem[NV]{NV}
C.~N{\u{a}}st{\u{a}}sescu and F.~van Oystaeyen, \emph{Graded ring
theory},
  North-Holland Publishing Co., Amsterdam, 1982.

\bibitem[Ro]{Ro}
D.\ Rogalski, \emph{GK-dimension of birationally commutative
surfaces}, preprint, to appear in Trans. Amer. Math. Soc.,
arXiv:0707.3643.

\bibitem[RoSm]{RoSm}
Louis~H. Rowen and Lance Small, \emph{Primitive ideals of algebras},
Comm.
  Algebra \textbf{25} (1997), no.~12, 3853--3857.

\bibitem[RS1]{RS1}
D.\ Rogalski and J.\ T.\ Stafford, \emph{A class of noncommutative
projective surfaces}, to appear in Proc. London Math. Soc.,
arXiv:math/0612657.

\bibitem[RS2]{RS2}
D.~Rogalski and J.~T. Stafford, \emph{Na\"\i ve noncommutative
blowups at
  zero-dimensional schemes}, J. Algebra \textbf{318} (2007), no.~2, 794--833.

\bibitem[RZ]{RZ}
D.~Rogalski and J.~J.~Zhang, \emph{Canonical maps to twisted rings},
Math. Z. \textbf{259} (2008), no. 2, 433--455.

\bibitem[Si]{Si}
S. J. Sierra, \emph{Geometric idealizers}, preprint, arXiv:0809.3971.

\bibitem[Sm]{Sm}
Lance~W. Small, \emph{Prime ideals in {N}oetherian {${\rm
PI}$}-rings}, Bull.
  Amer. Math. Soc. \textbf{79} (1973), 421--422.


\bibitem[Sp]{Sp}
T.~A. Springer, \emph{Linear algebraic groups}, second ed., Progress
in Mathematics, vol.~9, Birkh\"auser Boston Inc., Boston, MA, 1998.



\end{thebibliography}
\end{document}